\documentclass[a4paper,11pt]{amsart}
\usepackage[english]{babel}

\usepackage{amsmath,amscd,amssymb,amsthm,amsxtra,enumerate}
\usepackage[foot]{amsaddr}
\usepackage{mathrsfs}
\usepackage{mathtools}
\usepackage{tikz}
\usepackage{color,graphicx}
\usepackage{epsfig,epstopdf}
\usepackage[margin=1.25in]{geometry}
\usepackage{enumerate}
\usepackage{hyperref}

\newcommand{\norm}[1]{\left\Vert#1\right\Vert}

\newcommand{\R}{\mathbb{R}}
\newcommand{\re}{\mathbb{R}}
\newcommand{\se}{\mathbb{S}}

\newcommand{\Tau}{\mathcal{T}}

\newcommand\reallywidehat[1]{\arraycolsep=0pt\relax%
\begin{array}{c}
\stretchto{
  \scaleto{
    \scalerel*[\widthof{\ensuremath{#1}}]{\kern-.5pt\bigwedge\kern-.5pt}
    {\rule[-\textheight/2]{1ex}{\textheight}} 
  }{\textheight} %
}{0.5ex}\\           
#1\\                 
\rule{-1ex}{0ex}
\end{array}
}

\newcommand{\cN}{\mathcal{N}}
\newcommand{\cP}{\mathcal{P}}
\newcommand{\cPn}{\mathcal{NSP}}
\newcommand{\cQ}{\mathcal{Q}}

\newcommand{\vep}{\varepsilon}

\DeclareMathOperator{\dist}{dist}
\DeclareMathOperator{\dvie}{div}

\DeclareMathOperator{\dive}{div}

\newtheorem{thm}{Theorem}[section]

\newtheorem{lem}[thm]{Lemma}

\theoremstyle{definition}

\newtheorem{rem}[thm]{Remark}

\numberwithin{equation}{section}

\allowdisplaybreaks

\author[A. Biswas]{Animesh Biswas}

\address{Department of Mathematics \\
University of Nebraska-Lincoln \\
210 Avery Hall, Lincoln\\
NE 68588, United States of America}
\email{abiswas2@unl.edu}

\author[M. Foss]{Mikil Foss}

\email{mikil.foss@unl.edu}

\author[P. Radu]{Petronela Radu}

\email{pradu@unl.edu}

\thanks{Research partially supported by,  }

\begin{document}

\title[Nonlocal Curvature]{Nonlocal Mean Curvature with Integrable Kernel}
\maketitle

\begin{abstract}
We study the prescribed constant mean curvature problem in the nonlocal setting where the nonlocal curvature has been defined as
$$
H^J_{\Omega}(x):=\int_{\mathbb{R} ^n} J(x-y)(\chi_{\Omega^c}(y)-\chi_{\Omega}(y))dy,
$$
where $x \in \mathbb{R}^n$, $\Omega \subset \mathbb{R}^n$, $\chi$ is the characteristic function for a set, $J$ is a radially symmetric, nonegative, nonincreasing convolution kernel. Several papers have studied the case of nonlocal curvature with nonintegrable singularity, a generalization of the classical curvature concept, which requires the regularity of the boundary to be above $C^2$.  Nonlocal curvature of this form appears in many different applications, such as image processing, curvature driven motion, deformations. In this work, we focus on the problem of constant nonlocal curvature defined via integrable kernel. Our results offer some extensions to the constant mean curvature problem for nonintegrable kernels, where counterparts to Alexandrov’s theorem in the nonlocal framework were established independently by two separate groups: Ciraolo, Figalli, Maggi, Novaga, and respectively, Cabr\'e, Fall, Sol\`a-Morales, Weth. Using the nonlocal Alexandrov's theorem we identify surfaces of constant mean nonlocal curvature for different integrable kernels as unions of balls situated at distance $\delta$ apart, where $\delta$ measures the radius of nonlocal interactions. 
\end{abstract}

\providecommand{\keywords}[1]
{
  \small	
  {\textit{Keywords:}} #1
}


\providecommand{\subjclass}[1]
{
  \small	
  {\textit{MSC2010:}} #1
}

\section{Introduction}
In this paper, we study the constant nonlocal mean curvature problem, i.e. find the surfaces for which the curvature is constant. After the seminal paper~\cite{Caffa-Roq-Savin} which introduced the concept of nonlocal curvature, there has been increasing interest in problems involving this new concept. Given $\Omega\subseteq\re^n$ with a sufficiently smooth boundary $\partial\Omega$ and $0<s<1$, the authors Caffarelli, Roquejofre, and Savin of~\cite{Caffa-Roq-Savin} introduced
\begin{align}\label{eq:nlc_caffa}
    H_\Omega^s(x) = \int_{\re^n}\frac{\chi_{\Omega^c}(y)-\chi_{\Omega}(y)}{|x-y|^{n+2s}} \, dy
\end{align}
as a nonlocal mean curvature of $\Omega$ at $x \in \partial \Omega$. Here, we use $\chi_E$ to denote the characteristic function of $E\subseteq\re^n$. The kernel of interaction that appears in the definition of $H^s_\Omega$ was chosen to be the strongly singular function $J(x-y)$, with $J(z)=|z|^{-n-2s}$ for $z\in\re^n\setminus\{0\}$, and $0<s<1$ (observe that $J\notin L^1(\re^n)$). In~\cite{curvaturepaper}, Maz\'{o}n, Rossi, and Toledo extended the concept of nonlocal curvature to include integrable kernels, a case which is of interest to us as well. In particular, our integrability assumption has a prototypical example given by 
\begin{equation}\label{eq:exam_kernel}
   z\mapsto\frac{\chi_{B_{r}}(z)}{|z|^{n-\alpha}},
\end{equation}
for some $\alpha>0$ and $r>0$. Here $B_{r}\subseteq\re^n$ is the open ball with radius $r$ centered at the origin. More generally, we assume the integrable kernel $J\in L^1(\re^n)$ satisfies the following
\begin{itemize}
    \item (\textbf{rotational symmetry}) there exists $\mu:\re\to\re$ such that $J(z)=\mu(|z|)$ for all $z\neq0$,
    \item (\textbf{compact support}) there exists $r>0$ such that $\mu(\rho)=0$ for all $\rho\ge r$,
    \item (\textbf{radially decreasing}) if $0<r_2<r_1\le r$, then $\mu(r_1)<\mu(r_2)$.
\end{itemize}
Given a measurable set $\Omega \subseteq \re^n$, the nonlocal curvature at $x \in \re^n$, is defined as
\begin{equation}\label{curvat}
    H^J_{\Omega}(x)
    :=\int_{\re ^n} J(x-y)(\chi_{\Omega}(y) - \chi_{\Omega^c}(y))
    =\int_{\re ^n} J(x-y)\tau_\Omega(y)dy.
\end{equation}
For convenience, we will use $\tau_\Omega = \chi_{\Omega}- \chi_{\Omega^c}$. Since $J$ is integrable and $\tau_\Omega\in L^\infty(\re^n)$, we see that $H_\Omega^J$ is well-defined and finite on all of $\re^n$ without any regularity assumptions on $\partial\Omega$.

\subsection{Motivation} In \cite{Caffa-Roq-Savin, Osher}, it was shown that the classical mean curvature problem appears in the cellular automata problem in the following way. Fix $\vep>0$, and set $t_0=0$ and $\Omega_0=\Omega$. For each $k\in\cN$, put $t_k=k\vep$. We iteratively define $\Omega_k=\{x\in\re^n: u_k(t_k,x)\ge1/2\}$, where $u_k:[t_{k-1},t_k]\times\re^n\to\re$ is the solution to the heat equation $\partial_t u_k - \Delta u_k = 0$ with initial condition $u_k(t_{k-1},\cdot) = \chi_{\Omega_{k-1}}$. The sequence of surfaces $\{\partial \Omega_k\}_{k=0}^\infty$ provide a discrete approximation to the flow of $\Omega$ by mean curvature. A natural candidate for the flow by nonlocal mean curvature uses the fractional heat equation $u_t-(-\Delta)^s u=0$, where $0<s<1$ is fixed. Here the fractional Laplacian operator, $(-\Delta)^s$ is defined by
$$(-\Delta)^s u = \mathcal{F}^{-1}\left( |\xi|^{2s} \mathcal{F}u \right), $$
where $\mathcal{F}$ is the Fourier transform and $\mathcal{F}^{-1}$ is its inverse. The fractional Laplacian operator and in general fractional elliptic and parabolic operators have been very well-studied, see for example \cite{BDS, B-S, Caffa-Silv, Caffarelli-Silvestre-Salsa, Stinga-Caffa}.

In 1951 \cite{Caccio1}, Caccioppoli gave the following definition of perimeter. The perimeter of $\Omega$ inside a measurable set $E\subseteq\re^n$ is defined as
$$Per(\Omega, E) = \sup \bigg\{ \int_E \chi_\Omega \dive \phi \, dx: \phi \in C^\infty_0(E, \R^n), |\phi| \leq 1 \bigg\}, $$
a formulation that is also known as the total variation of the function $\chi_\Omega$. If $\partial\Omega$ is $C^2$, it is easy to show that $Per(\Omega, \R^n) = \mathcal{H}^{n-1}(\partial \Omega)$. In \cite{Caffa-Roq-Savin}, the authors defined the nonlocal energy functional as, when $s<1/2$, 
$$\norm{u}^2_{H^s} = \int_{\R^n} \int_{\R^n} \frac{(u(x) - u(y))^2}{|x-y|^{n+2s}} dx dy.$$
 In this paper the authors showed that minimizing this energy functional of a characteristic function $\chi_\Omega$ (which can be interpreted as a nonlocal perimeter of $\Omega$), inside a bounded set $E$ under given nonlocal boundary condition, one can obtain the minimal surfaces $S = \partial \Omega$ whose Euler-Lagrange equation is 
$$\int_{\R^n} (\chi_\Omega(y) - \chi_{\Omega^c}(y))|x-y|^{-n-2s} dy= 0, \quad \text{for}~x \in S.  $$
The above-mentioned quantity, 
$$-\int_{\R^n} (\chi_\Omega(y) - \chi_{\Omega^c}(y))|x-y|^{-n-2s} dy $$
is labeled as the nonlocal curvature. Introduction of the nonlocal minimal surfaces and mean curvature motivated a rich body of literature in the last ten years. In 2018, two different groups, \cite{Cabre-Fall-Sola-Weth,Figalli} studied the constant nonlocal mean curvature problem, as defined above. Both groups used the idea of the Alexandrov's moving plane in the nonlocal setting and proved that a ball is the only solution to the constant nonlocal mean curvature problem. Similar type of constant nonlocal mean curvature problems in different settings were studied in \cite{Cabre-Fall-Weth, Cabre-Fall-Weth1}. Additionally, we mention some works on nonlocal minimal graphs, \cite{Cabre-Cozzi} and nonlocal curvature flows \cite{Valdinocc1}. For other results concerning this particular type of nonlocal curvature, see \cite{Valdinocc2, Valdinocc3}.

The motivation for the introduction and study of the concepts of perimeter and curvature in the nonlocal framework with integrable kernel is born from similar considerations as for other models, mainly, {\it to eliminate restrictions on the smoothness of the geometry (boundary) of the domain}. The example from image processing \cite{cinti2019quantitative,curvaturepaper} provides a simple illustration for considering new geometrical measures for domains in the nonlocal setting.

 For boundaries that are smooth the nonlocal curvature was shown to provide a good approximation to the classical curvature \cite{curvaturepaper}. However, note that the nonlocal curvature has the advantage of being well-defined at corners or other singular points of the boundary. 
 Moreover, the definition makes sense for all points of a set, not only on a boundary, a feature which may be useful for irregular domains (such as cusps) as one may be able to capture the rate of ``narrowing" as the cusp, or other singular point, is approached. An important application of the nonlocal curvature can be found for domains with ``zig-zag" boundaries where the classical curvature could be defined only piecewise on segments; however, as the segments get smaller (and possibly approach a smooth boundary), there is no measure for how curved or bent the boundary is. 
 
Motivated by the significance of this version of nonlocal curvature, we are interested in identifying surfaces with a prescribed nonlocal curvature, where the nonlocal curvature is defined as in \eqref{curvat}. In other words, for given $f$ we are interested in the solution $u$ of the equation 
\begin{align}\label{eq:prescribed_curv}
    \int_{\re ^n} J(x-y)(\chi_{\Omega^c}(y)-\chi_{\Omega}(y))dy = f(x), \quad \text{where}~ x \in \partial \Omega .
\end{align}
Above, $\Omega$ is a bounded set, and we will further assume that $f$ is a constant function. On the other hand, when $f$ is nonconstant we assume that $\partial \Omega$ is the graph of a function. Although nonlocal curvature can be defined for surfaces with no differentiability, for our results we will have a standing assumption that the boundary is at least $C^1$. Indeed, this is required by the main technique used for the constant curvature problem, the Alexandrov's moving plane method. Regarding this regularity assumption, we mention that very recently, the authors in \cite{Bucur1} were able to consider, for a particular type of kernel, boundaries that are not differentiable. More specifically, they considered the following problem. Let $\Omega$ a bounded set  which satisfies some non-degeneracy conditions. For any point $x \in \partial^* \Omega$, where $\partial^* \Omega$ is the essential boundary of $\Omega$, if
\begin{equation}\label{cond1}
    |B_r(x) \cap \Omega|=c,
\end{equation}
where $B_r(x)$ is the ball of radius $r$ centered at $x$, and $c>0$ and $r>0$ are fixed, then the domain $\Omega$ must be a union of balls. Note that \eqref{cond1} translates to constant mean curvature when the kernel is given by the identity function for $B_r(x)$. Although the authors in \cite{Bucur1} got the desired solution in a more general setup, our methods provide an alternative approach which may be valuable in handling different problems involving nonlocal curvature.

\subsection{Main contributions and significance}

Our main interest in this paper is to study the constant nonlocal mean curvature problem. For our first result, we will consider the case of an infinite radius of interaction, i.e. $r= \infty$. In addition, assume that the kernel of interaction $J$ is differentiable everywhere, except possibly at $x=0$, and strictly decreasing. The proof of the following theorem is very similar to the proofs given for the case of fractional Laplacian kernel given in \cite{Cabre-Fall-Sola-Weth, Figalli} with the exception that $J$ is integrable. Since $J$ is integrable near $x=0$, we assume some decay rate for $J$ and $\nabla J$ near $x=0$; specifically, we assume 
\begin{equation}\label{eq:J estimate}
\begin{cases}
|J(x)| \leq \min \left\{ \frac{C}{|x|^{n-\alpha}}, \frac{C}{|x|^{n+\alpha_1}} \right\} \\
|\nabla J(x) | \leq \min \left\{ \frac{C}{|x|^{n+1 -\alpha}}, \frac{C}{|x|^{n+1 +\alpha_1}} \right\},
\end{cases}
\end{equation}
for all $x \neq 0$ and for some $\alpha, \alpha_1 >0$.

\begin{thm}\label{thm:c2beta_bdd}
Let $J:\re^n \to \re$ be radially symmetric, strictly decreasing nonnegative function which is $C^\infty(\re^n \setminus \{0 \})$ and satisfies the decay estimates \eqref{eq:J estimate}. In addition, let $J(x) = \phi(|x|)$ for some function $\phi: [0, \infty) \to \re_+$ with $\phi'(t) <0$ for all $t>0$.
Let $\Omega$ be a nonempty bounded open set with  $C^{1, \beta}$  boundary for some $\beta = \min \{\alpha -1, 0 \} \geq 0$ and with the property that $H^J_\Omega$ is constant at every point on the boundary. Then $\Omega$ is a ball.
\end{thm}

Nest we consider the case when the kernel $J$ satisfies all the previous conditions except it has compact support, precisely inside the ball $B_r$. As $J$ is differentiable everywhere except possibly at $x=0$, $J=0$ on $\partial B_r$. Also, as $J=0$ in $(B_r)^c$, we don't need to consider the decay of $J$ near infinity. In this case, we assume  
\begin{equation}\label{eq:J estimate_2}
\begin{cases}
|J(x)| \leq  \frac{C}{|x|^{n-\alpha}} \\
|\nabla J(x) | \leq \frac{C}{|x|^{n+1 -\alpha}}
\end{cases}
\end{equation}
for all $x \neq 0$ for some $\alpha >0$.

Under these assumptions we obtain the following theorem:
\begin{thm}\label{thm:c2beta_bdd_2}
Let $J:\re^n \to \re$ be radially symmetric, strictly decreasing nonnegative function which is $C^\infty(\re^n \setminus \{0 \})$  and satisfies the estimates in \eqref{eq:J estimate_2}. In addition to that if $J(x) = \phi(|x|)$ for some function $\phi: [0, \infty) \to \re_+$ then $\phi'(t) <0$ for all $t>0$ and $J$ has a compact support in $B_r$.
Let $\Omega$ be nonempty bounded, open, set with  $C^{1, \beta}$  boundary for some $\beta = \min \{\alpha -1, 0 \} \geq 0$ and with the property that $H^J_\Omega$ is constant at every point on the boundary then $\Omega$ is a {union of balls which are at a distance at least $r$ from each other.}
\end{thm}

Our next theorem is for some $J$ which has compact support in $B_r$. $J$ is radially symmetric, strictly decreasing function which is in $C^\infty(B_r \setminus \{0\} )$. Inside the ball $B_r$, $J$ satisfies the estimate given in \ref{eq:J estimate_2}. But $J$ has a jump at the boundary $\partial B_r$. In addition, if $J(x) = \phi(|x|)$ for some $\phi: [0, \infty) \to \re_+$, then $\phi'(t) < 0$ in $(0, r)$. As an example, one can consider the kernel given in \eqref{eq:exam_kernel}. Next we impose the following condition on the boundary of the set $\Omega$. We assume the $C^1$-norm of the boundary is uniformly bounded and small so that $\partial \Omega$ does not vary wildly inside $B_{r+\delta}(x)$ for any $x \in \partial \Omega$ and for some $\delta >0$.  Precisely we assume that for every $x \in \partial \Omega$, $\partial \Omega$ is a graph of some $C^1$ function $f^x$ inside the ball $B_{r+\delta}(x)$ (for some $\delta>0$) such that $\norm{\nabla f^x} \leq M <1$.

\begin{thm}\label{thm:c2beta_bdd_3}
Let $J:\re^n \to \re$ be radially symmetric, strictly decreasing nonnegative function which is $C^\infty(B_r \setminus \{0\} )$  and satisfies the estimates in \eqref{eq:J estimate_2}. In addition to that if $J(x) = \phi(|x|)$ for some function $\phi: [0, \infty) \to \re_+$ then $\phi'(t) <0$ in $(0, r)$ and $\phi$ has a jump at $r$ from a positive value to $0$.
Let $\Omega$ be nonempty bounded open connected set with  $C^{1, \beta}$  boundary for some $\beta = \min \{\alpha -1, 0 \} \geq 0$. In addition to that we assume that the boundary has uniformly bounded $C^1$-norm as mentioned above. Next if $H^J_\Omega$ is constant at every point on the boundary then $\Omega$ is a {union of balls which are at a distance at least $r$ from each other.}
\end{thm}

In all the theorems given above $J$ is strictly decreasing in its support set. Now we present a result when $J$ is not necessarily strictly decreasing in its support set. In particular, we  assume that $J(x) = \chi_{B_r}$. But we don't have a sphere solution for all $r$. To see that, consider any set with diameter $< r/2$. Then using the definition of nonlocal curvature at some point $x_1 \in \partial \Omega$ is
$$H^J_\Omega(x_1) = |B_r(x_1) \cap \Omega^c| - |B_r(x_1) \cap \Omega| = |B_r| - 2|B_r(x_1) \cap \Omega)| = |B_r| - 2|\Omega|.$$
The above expression is independent of $x_1$ and hence there is no unique solution. As mentioned before, in \cite{Bucur1}, the authors studied this problem when the boundary of the set is not smooth. They assumed a non-degeneracy condition on the set, which says that if $\Omega$ is $r$-degenerate then
$$\inf_{x_1, x_2 \in \partial \Omega^*} \frac{|\Omega \cap (B_r(x_1) \Delta B_r(x_2))|}{\norm{x_1-x_2}} =0.$$
Our aim is to present a different proof of the similar result under uniform $C^1$-norm boundary condition stated before the Theorem \ref{thm:c2beta_bdd_3}. 
\begin{thm}\label{thm:thm:c2beta_bdd_4}
Let $J(x) = \chi_{B_r}(x)$. 
Let $\Omega$ be nonempty connected bounded open set which has  $C^{1}$  boundary with similar condition as in Theorem \ref{thm:c2beta_bdd_3} . It also satisfies the constant mean curvature conditions, that is $H^J_\Omega$ to be constant at every point on the boundary. Then $\Omega$ is a ball.
\end{thm}

Before we present the proof of Theorems \ref{thm:c2beta_bdd}, \ref{thm:c2beta_bdd_2}, \ref{thm:c2beta_bdd_3} and \ref{thm:thm:c2beta_bdd_4} let us first describe the idea of the proof. In the case of a highly singular kernel, see \cite{Cabre-Fall-Sola-Weth, Figalli}, where the authors introduce and prove a nonlocal version of the celebrated Alexandrov's Theorem, {\cite{Alexandrov}}. Our proof will follow a similar approach. For simplicity consider the kernel whose support set is $\R^n$. Next we consider a hyperplane, $HP_e$, which is normal to a unit vector $e \in \se^{n-1}$. One can see that $HP_e$ partitions $\Omega$ into two disjoint sets, say $\Omega_1, \Omega_2$ (either of them can be empty). 
 Let us start with a situation when one of them, suppose $\Omega_2$, is empty. Let $R_e$ be reflection operator with respect to $HP_e$ at that position. In that scenario, it is trivially true that $R_e(\Omega_2) \subset \Omega_1$.
 Then as we move $HP_e$ (translation in the $e$ direction), $\Omega_2$ starts to become nonempty still satisfying $R_e(\Omega_2) \subset \Omega_1$. Here we are using same notation for the quantities $\Omega_1, \Omega_2$ and $R_e$ which, in principle, depend also on the position of $HP_e$.
 Next after a particular translation, $R_e(\Omega_2)$ does not remain a subset of $\Omega_1$. Let us name that translation of $HP_e$ as the critical hyperplane $HP^*_e$ and the corresponding reflection operator as $R^*_e$. Next, we prove that $\Omega \Delta R^*_e(\Omega)$ is an empty set, where 
 $$\Omega \Delta R^*_e(\Omega) = (\Omega \setminus R^*_e(\Omega) ) \cup ( R^*_e(\Omega) \setminus \Omega)$$
 That implies $\Omega$ is symmetric with respect to the $e$-direction.
 Proof of the above claim, in case of nonlocal curvature, depends on the strict monotonicity, size conditions and support of $J$ along with the boundary regularity of $\Omega$.  Finally as this claim is true for any $e \in \se^{n-1}$, we must have that $\Omega$ is a ball.

 As mentioned before, our focus is on the study of the constant nonlocal mean curvature problem when the boundary is at least $C^1$. Very recently, the paper \cite{Bucur1} solved the constant mean curvature problem for boundaries with less regularity (even measurable), our work provides a different perspective and approach to this problem. Specifically, for different integrable kernels we investigate the differentiability of the curvature function on the boundary. To this end, regularity properties of the boundary are required. Establishing differentiablity of the curvature function is a critical step in solving the constant mean curvature problem and is an independent problem deserving attention on its own.

 \section{Proof of Theorem \ref{thm:c2beta_bdd}}
Proof of Theorem \ref{thm:c2beta_bdd} is very similar to the proof given in \cite{Cabre-Fall-Sola-Weth, Figalli}. Since the main idea of this proof will be used for the other related theorems, we give a brief outline of the proof for the convenience of the reader. 
As mentioned before, the principal idea of the proof is the Alexandrov's method. We discuss that with details in the proof. Another important step is to prove that $H^J_\Omega \in C^1(\partial \Omega)$. This is in parallel to a result given in \cite[Lemma 2.1]{Figalli} and \cite[Proposition 2.1]{Cabre-Fall-Sola-Weth}. As the proof of this result is very similar except some minor modifications due to the regularity of the boundary, we present the proof of that result in the Appendix. We provide only the statement of that result in this section.

\begin{lem}\label{lem:Figalli 2.1}
Suppose $\phi_\varepsilon \in C^\infty_c([0, \infty))$ is such that $\phi_\varepsilon \geq 0, \phi'_\varepsilon \leq 0$ and it satisfies the following estimates, for all $t>0$
$$
\begin{cases}
\max\{t^{n-\alpha}, t^{n+\alpha_1} \} \phi_\varepsilon(t)  +  \max\{t^{n-\alpha+1}, t^{n+\alpha_1+1} \} |\phi'_\varepsilon(t)| \leq C(n, \alpha, \alpha_1) \\
|\phi'_\varepsilon(t)| \to |\phi'(t)| \quad \text{as}~ \varepsilon \to 0 \\
\phi_\varepsilon(t) = \phi(t) \quad \text{if}~ t> \varepsilon 
\end{cases}
$$
where $\phi(t)$ has been defined in Theorem \ref{thm:c2beta_bdd}.
Then we define 
$$H^\varepsilon_\Omega (x) = \int_{\re^n} \tau_\Omega (y) \phi_\varepsilon(|x-y|) dy,  $$
$x \in \re^n$. Next we assume that $\Omega$ is nonempty bounded open set with  $C^{1, \beta}$  boundary for some $\beta > 0$ such that $\beta + \alpha >1$. Then $H^J_\Omega \in C^1(\partial \Omega)$ and $H^\varepsilon_\Omega \to H^J_\Omega$ in $C^1(\partial \Omega)$ as $\varepsilon \to 0$.
\end{lem}
Proof of the Lemma is given in the Appendix. Now we can write using Lemma \ref{lem:Figalli 2.1}
$$\nabla H^J_\Omega(x) \cdot e = \lim_{\varepsilon \to 0} \nabla H^\varepsilon_\Omega(x) \cdot e$$
for any $x \in \partial \Omega$ and $e \in T_x(\partial \Omega)$. 
Next we prove Theorem \ref{thm:c2beta_bdd}. But before we do, we want to introduce some notation which will be useful in later sections also. 
For $e \in \se^{n-1}$, $E \subset \re^n$, $\tau \in \re$, we define 
\begin{align*}
    &\pi_\tau = \{ x \in \re^n : x \cdot e = \tau \}, \quad \hbox{plane perpendicular to $e$-vector} \\
    &(\pi_\tau)_+ = \{ x \in \re^n : x \cdot e > \tau \},\\
    &E_\tau = E \cap (\pi_\tau)_+, \\
    &R_\tau(x)  = x- 2(x \cdot e - \tau) e, \\
       &\quad \hbox{where $R_\tau$ is the reflection operator with respect to $\pi_\tau$ plane}, \\
    &R_\tau(E) = \{ R_\tau(x): x \in E \}. 
\end{align*}
Next we define, for a bounded open set $\Omega$, $\mu = \sup \{ x \cdot e : x \in \Omega \}$. If the set $\Omega$ has a $C^1$ boundary, \cite{Figalli}, then for every $\tau < \mu$, sufficiently close to $\mu$, $R_\tau(\Omega_\tau) \subseteq \Omega$. Therefore we define
$$\lambda: = \inf \{\tau \in \re : R_{\tilde{\tau}}(\Omega_{\tilde{\tau}}) \subset \Omega \quad \hbox{for all}~ \tilde{\tau} \in (\tau, \mu) \}. $$

We denote $\pi_\lambda$ to be the critical hyperplane for a fixed direction $e$. 
For our convenience, we use $R$ in place of $R_\lambda$,  in the following discussion, to denote reflection with respect to the critical hyperplane.
Next from \cite{Alexandrov}, we see that for any direction $e$, at least one of the following two conditions will hold,
\begin{enumerate}[(a)]
    \item interior touching: $\partial R(\Omega)$ is tangent to $\partial \Omega$ at some point $x_0$ which is the refection of a point $x'_0 \in \partial R(\Omega) \setminus \pi_\lambda $.
    \item non-transversal intersection: $\pi_\lambda$ is orthogonal to $\partial \Omega$ at some point $x_0 \in \partial \Omega \cap \pi_\lambda$.
\end{enumerate}

\begin{proof}[Proof of Theorem \ref{thm:c2beta_bdd}]
As mentioned before, proof is similar as in \cite{Cabre-Fall-Sola-Weth, Figalli}. Hence we give a brief outline where we mention the important steps only. 
We start by considering any $e \in \se^{n-1}$.  For a fixed $e \in \se^{n-1}$, we will have either case $(a)$ or case (b). Then we want to prove, irrespective of case (a) or (b), that $\Omega \setminus R(\Omega) = R(\Omega) \setminus \Omega = \varnothing$, where $\varnothing$ is the null set. Since this is true for any $e$, $\Omega$ is symmetric in any direction and hence is a ball. 
Assume that, without the loss of generality, $e = \langle 1, 0, 0, \cdots 0 \rangle$. If $\pi_\lambda$ is the critical hyperplane in this direction for some $\lambda$, then the rotation operator with respect to $\pi_\lambda$, is given by $R(x) = (-x_1+2 \lambda , x_2, \cdots, x_n)$ for any $x=(x_1, x_2, \cdots, x_n)$. 

     Case (a): We notice that $x_0, x'_0=R(x_0) \in \partial \Omega \cap \partial R(\Omega)$. Then 
    $$H^J_\Omega(x_0) - H^J_\Omega(R(x_0)) = 0.$$ 
    Using $\tau_\Omega(y) = (\chi_{\Omega^c}(y) - \chi_\Omega(y))$ for any $y \in \re^n$,
    \begin{align*}
        H^J_\Omega (R(x_0)) &= -\int_{\re^n} \tau_\Omega(y) J(R(x_0)-y) dy \\
        &=-\int_{\re^n}\tau_\Omega(y) J(-(x_0)_1+2 \lambda -y_1, (x_0)_2-y_2, \cdots, (x_0)_n-y_n) dy \\
        &= -\int_{\re^n}\tau_\Omega(y) J(-(x_0)_1 +(-y_1+2 \lambda), (x_0)_2-y_2, \cdots, (x_0)_n-y_n) dy \\
        &= -\int_{\re^n}\tau_\Omega(y) J((x_0)_1 -(-y_1+2 \lambda), (x_0)_2-y_2, \cdots, (x_0)_n-y_n) dy \\
             &\qquad \hbox{(using radial symmetry of $J$)}\\
        &= -\int_{\re^n}\tau_\Omega(y) J(x_0-R(y))dy \\
        &= -\int_{\re^n} \tau_{R(\Omega)}(y) J(x_0-y) dy \quad \hbox{(doing a change of variable)} \\
        &= H^J_{R(\Omega)}(x_0)
    \end{align*}
    Next we observe that $\tau_\Omega = -\tau_{R(\Omega)}=1$ in $\Omega \setminus R(\Omega)$ and $\tau_\Omega = -\tau_{R(\Omega)}=-1$ in $R(\Omega) \setminus \Omega$. Then we have, 
    \begin{align}\label{eq: interior}
        0 &= H^J_\Omega(x_0) - H^J_\Omega(R(x_0)) \\ \nonumber
        &= H^J_\Omega(x_0) - H^J_{R(\Omega)}(x_0)\\ \nonumber
        &= \int_{\re^n} \bigg( \tau_{\Omega}(y) - \tau_{R(\Omega)}(y)\bigg)J(x_0 - y) dy\\ \nonumber
        &= \int_{\Omega \setminus R(\Omega)} J(x_0-y) dy - \int_{R(\Omega) \setminus \Omega} J(x_0 -y) dy \\ \nonumber
        &= \int_{\Omega \setminus R(\Omega)} \big(J(x_0-y)-J(x_0 - R(y)) dy
    \end{align}
    From the construction, $x_0 \in \Omega \setminus R(\Omega)$ and hence $|x_0 -y| \leq |x_0 - R(y)|$ for all $y \in \Omega \setminus R(\Omega)$. 
    Since $J$ is a strictly decreasing function, we have for all $y \in \Omega \setminus R(\Omega)$, $ \big(J(x_0-y)-J(x_0 - R(y))$ is a positive quantity. That implies $|\Omega \setminus R(\Omega) | =0$ and similarly $|R(\Omega) \setminus \Omega| =0$.
    
     Case (b): We see that $\pi_\lambda$ is orthogonal to $\partial \Omega$, which implies that the vector $e = \langle 1, 0, \cdots, 0 \rangle$ is tangent to $\partial \Omega$ at the point $x_0$. Similarly, $e$ is tangent to $\partial R(\Omega)$ at $x_0$. Since $H^J_\Omega$ is a constant function on $\partial \Omega$, $H^J_{R(\Omega)}$ is constant on $\partial R(\Omega)$. So the tangential derivative of $H^J_\Omega$, at the point $x_0$ will be zero, i.e $\partial_e H^J_\Omega(x_0) = \nabla H_\Omega(x_0) \cdot e = 0$. Using similar reasons, $\partial_e H^J_{R(\Omega)}(x_0) = 0$. From Lemma \ref{lem:Figalli 2.1}, we know that 
    $\partial_e H^J_{\Omega}(x_0) = \lim_{\varepsilon \to 0} \partial_e H^\varepsilon_{\Omega}(x_0)$.
    As $H^\varepsilon_J(x_0) = \int_{\re^n} \tau_\Omega (y) \phi_\varepsilon(|x_0-y|) dy,$ denoting $J_\varepsilon(x) = \phi_\varepsilon(|x|),$    we can write
    $$ \partial_e H^J_\varepsilon(x_0) = \int_{\re^n} \tau_\Omega (y) \nabla J_\varepsilon(|x_0-y|) \cdot e \, dy = \int_{\re^n} \tau_\Omega (y) \phi'_\varepsilon(|x_0-y|) \frac{x_0-y}{|x_0-y|} \cdot e \, dy $$
    Since $e=\langle 1, 0, \cdots, 0 \rangle$,
    $$\partial_e H^J_\varepsilon(x_0) = \int_{\re^n} \tau_\Omega (y) \phi'_\varepsilon(|x_0-y|) \frac{(x_0)_1-y_1}{|x-y|} \, dy $$
Again using the fact $\partial_e H^J_\Omega(x_0) - \partial_e H^J_{R(\Omega)}(x_0)=0$, we have
    \begin{align*}
        0 &= \lim_{\varepsilon \to 0} \int_{\re^n} \big(\tau_\Omega (y) -\tau_{R(\Omega)}(y)) \phi'_\varepsilon(|x_0-y|) \frac{(x_0)_1-y_1}{|x_0-y|} \, dy \\
        &= \lim_{\varepsilon \to 0} \bigg[ \int_{\Omega \setminus R(\Omega)} \phi'_\varepsilon(|x_0-y|) \frac{(x_0)_1-y_1}{|x_0-y|} \, dy - \int_{R(\Omega) \setminus \Omega} \phi'_\varepsilon(|x_0-y|) \frac{(x_0)_1-y_1}{|x_0-y|} \, dy \bigg]
    \end{align*}
    Next we observe that $(x_0)_1 = \lambda$ and if $y\in \Omega \setminus R(\Omega)$ then $y_1<\lambda$. On the other hand, if $y \in R(\Omega) \setminus \Omega$ then $y_1> \lambda$. Hence
    \begin{equation}\label{eq:non transversal}
        0 = \lim_{\varepsilon \to 0} \bigg[ \int_{\Omega \setminus R(\Omega)} \phi'_\varepsilon(|x_0-y|) \frac{|\lambda-y_1|}{|x_0-y|} \, dy + \int_{R(\Omega) \setminus \Omega} \phi'_\varepsilon(|x_0-y|) \frac{|\lambda - y_1|}{|x_0-y|} \, dy \bigg]
    \end{equation}
    Since $\phi'_\varepsilon(t) \leq 0$ for all $t$ and $ \phi_\varepsilon(t) = \phi(t)$ when $t>\varepsilon$, then $\phi'_\varepsilon(|x_0-y|)<0 $ when $|x_0-y|>\varepsilon$ for any $y \in \re^n$. After defining, 
    $$E^1_{\varepsilon} = \{ y \in \Omega \setminus R(\Omega) : |x_0-y| > \varepsilon \}, E^2_{\varepsilon} = \{ y \in R(\Omega) \setminus \Omega : |x_0-y| > \varepsilon \},  $$
    we see that $|E^1_\varepsilon| = |E^2_\varepsilon|=0$ for every $\varepsilon>0$. That gives us $|\Omega \setminus R(\Omega)| = |R(\Omega) \setminus \Omega| = 0$.
Then we see that, in both cases $|\Omega \setminus R(\Omega)| = |R(\Omega) \setminus \Omega| = 0$. But $\Omega$ has $C^{1,\beta}$ boundary and hence $\Omega = R(\Omega)$. 
\end{proof}

\section{Proof of Theorem \ref{thm:c2beta_bdd_2}}
Since $J$ is supported on a finite set, precisely on $B_r$, the solution is not always a ball. If we impose a condition that $\Omega$ is a connected set, then the solution is a ball otherwise in general it will be a union of balls of same size, where the balls are at least $r$ distance away from each other. In the previous case, as the kernel is supported on $\re^n$, the point $x_0$ (interior touching point or point of non-transversal intersection) `can see' the whole $\Omega \setminus R(\Omega)$. But here, it `can see' only the part of $\Omega \setminus R(\Omega)$ which is inside $B_r(x_0)$. So here we may have different critical hyperplanes, for the same direction, with respect to which different connected components of $\Omega$ will be symmetric. We also need to use a covering type argument since $\Omega$ is bounded. 

\begin{proof}
Since $J$ is $C^\infty$ everywhere in $\re^n$ except at $x=0$, and is radially strictly decreasing to $0$ in $B_r$, we can repeat the computation given in Lemma \ref{lem:Figalli 2.1} and get a similar lemma which proves that $H^J_\Omega$ is in $C^1(\partial \Omega)$. For that, we need to define a function $\phi:[0, \infty) \to \re_+$ such that $J(x) = \phi(|x|)$. Then we know that $\phi'(t)<0$ in $(t, r)$ for all $0<t<r$.
Suppose $\phi_\varepsilon \in C^\infty_c([0, \infty))$ is such that $\phi_\varepsilon \geq 0, \phi'_\varepsilon \leq 0$ and it satisfies the following estimates, for all $t>0$
$$
\begin{cases}
t^{n-\alpha} \phi_\varepsilon(t)  +  t^{n-\alpha+1} |\phi'_\varepsilon(t)| \leq C(n, \alpha, \alpha_1) \\
|\phi'_\varepsilon(t)| \to |\phi'(t)| \quad \text{as}~ \varepsilon \to 0 \\
\phi_\varepsilon(t) = \phi(t) \quad \text{if}~ t> \varepsilon 
\end{cases}
$$
If we define 
$$H^\varepsilon_\Omega (x) = \int_{\re^n} \tau_\Omega (y) \phi_\varepsilon(|x-y|) dy, \quad x \in \re^n $$
then $H^J_\Omega \in C^1(\partial \Omega)$ and $H^\varepsilon_\Omega \to H^J_\Omega$ in $C^1(\partial \Omega)$ as $\varepsilon \to 0$ using the very similar steps as in the proof of Lemma \ref{lem:Figalli 2.1}.
Next we fix a direction vector $e$. Without loss of generality, we assume that $e = \langle 1, 0, \cdots, 0 \rangle$. Then Alexandrov's moving plane method provides the two situations with either interior touching or non-transversal intersection. Let $x_0$ be one such point in $\partial \Omega$ and $\pi_\lambda$ be the critical hyperplane corresponding to the direction $e$. We will denote such type of point as `point of contact'. Let $R_\lambda$ be the reflection operator with respect to $\pi_\lambda$, which will be denoted as operator $R$ for simplicity. Next we define,
$${\pi_\lambda}_+ = \{x: x_1 > \lambda \} \quad \text{and}~ \quad  {\pi_\lambda}_- = \{x: x_1 < \lambda \}$$
We also define
$$
\begin{cases}
\Omega_+ = \Omega \cap {\pi_\lambda}_+, \quad \text{and}~\quad \Omega_- = \Omega \cap {\pi_\lambda}_- \\
\partial \Omega_+ = \partial \Omega \cap {\pi_\lambda}_+, \quad \text{and}~\quad \partial \Omega_- = \partial \Omega \cap {\pi_\lambda}_-
\end{cases}
$$

 We know that if $x_0$ is an interior touching point then $x_0 \in \partial \Omega_-$. On the other hand, if $x_0$ is a non-transversal intersection then $x_0 \in \partial \Omega \cap \pi_\lambda$. Next if we repeat all the computations as in the proof of Theorem \ref{thm:c2beta_bdd}, we get the following two equations similar to \eqref{eq: interior} and \eqref{eq:non transversal},
 $$0 = \int_{\Omega \setminus R(\Omega)} \big(J(x_0-y)-J(x_0 - R(y)) dy $$
 and,
 $$  0 = \lim_{\varepsilon \to 0} \bigg[ \int_{\Omega \setminus R(\Omega)} \phi'_\varepsilon(|x_0-y|) \frac{|\lambda-y_1|}{|x_0-y|} \, dy + \int_{R(\Omega) \setminus \Omega} \phi'_\varepsilon(|x_0-y|) \frac{|\lambda - y_1|}{|x_0-y|} \, dy \bigg]$$
 where $y_1 = y \cdot e$. Now, in the interior case, $x_0 \in \partial \Omega_-$. Then for any $y \in \Omega \setminus R(\Omega)$, $|x_0 - y| < |x_0 - R(y)|$. As $J$ is strictly decreasing and $J$ has support in $B_r$, we have,
  \begin{align}\label{eq: bounded smooth kernel}
    |(\Omega \Delta R(\Omega))\cap B_r(x_0)|=0
\end{align}
On the other hand, in the non-transversal case, $\phi'_\varepsilon \leq 0$ for all $t \geq 0$. Since $\phi_\varepsilon(t) = \phi(t)$ when $t>\varepsilon$ then $\phi'_\varepsilon <0$ in the interval $(\varepsilon, r)$. Again defining, 
    $$E^1_{\varepsilon} = \{ y \in \Omega \setminus R(\Omega) : |x_0-y| > \varepsilon \}, E^2_{\varepsilon} = \{ y \in R(\Omega) \setminus \Omega : |x_0-y| > \varepsilon \},  $$
    we see that 
    $$|E^1_\varepsilon \cap B_r(x_0)| = |E^2_\varepsilon \cap B_r(x_0)|=0$$ for every $\varepsilon>0$. That gives us $|\Omega \setminus R(\Omega)| = |R(\Omega) \setminus \Omega| = 0$ inside $B_r(x_0)$ and hence \eqref{eq: bounded smooth kernel} is true.
From \eqref{eq: bounded smooth kernel}, we have 
\begin{equation}\label{eq:measure_zero}
    |(\Omega \setminus R(\Omega))\cap B_r(x_0)| = |(  R(\Omega) \setminus \Omega)\cap B_r(x_0)|=0
\end{equation}
As $\partial \Omega$ is $C^{1, \beta}$ for $\beta >0$ then for any $x_0 \in \partial \Omega$, $|B_\varepsilon(x_0) \cap \Omega| >0$ for every $\varepsilon>0$. Then \eqref{eq:measure_zero} can only happen, when $\Omega = R(\Omega)$ inside $B_r(x_0)$. Or we can say that their boundaries coincide inside $B_r(x_0)$. Let $\Omega$ be a union of disjoint connected components. Then $x_0$ lies on the boundary of one such component, say $\cN_x$. Again if $x_0$ is interior touching point then $x_0 \in \partial {\cN_x}_-$. On the other hand, if it is a non-transversal intersection point, then $x_0 \in \partial \cN_x \cap \pi_\lambda$. We already showed that, since $x_0$ is a point of contact, $\Omega = R(\Omega)$ in $B_r(x_0)$ and hence the boundaries coincide. If $\partial {\cN_x}_- \subseteq B_r(x_0)$ then we can say that 
\begin{equation}\label{eq:set symmetry}
    \Omega = R(\Omega), \quad \text{everywhere in }~ {\cN_x} .
\end{equation}

Otherwise, we choose a point $z_0 \in \partial {\cN_x}_-  \cap B_{r}(x_0))$ such that $|x_0-z_0| \geq r/2$. Notice that $z_0$ is a point of interior touching with respect to the hyperplane $\pi_\lambda$. Using similar analysis as in interior touching, we have 
  $$|(\Omega \setminus R(\Omega))\cap B_r(z_0)|= |(R(\Omega) \setminus \Omega) \cap B_r(z_0)| = 0$$ 
  that implies $\Omega  = R(\Omega)$ inside $B_r(z_0)$. Again if $ \partial {\cN_x}_- \subseteq B_r(x_0) \cup B_r(z_0)$ then we have \eqref{eq:set symmetry}. Otherwise, we choose a point $z_1 \in \partial {\cN_x}_- \cap B_{r}(z_0)$ such that $|x_0 - z_1|\geq r/2$ and $ |z_1- z_0| \geq r/2$. Likewise, in the next iteration, we choose a point $z_2 \in \partial {\cN_x}_- \cap B_{r}(z_1)$ such that $|z_2-x_0|>r/2, |z_2-z_0|>r/2$ and $|z_2-z_1|>r/2$ if $\partial {\cN_x}_-$ is not contained in the union of those balls, i.e, $B_r(x_0) \cup B_r(z_0) \cup B_r(z_1)$. This construction allows only finitely many iterations since $\partial \Omega$ and hence $\partial \cN_x$ is compact. Hence, we should get \eqref{eq:set symmetry}. That further implies that ${\cN_x}_+ = R({\cN_x}_-)$ and $\cN_x$ is symmetric about $\pi_\lambda$. 
  
  Next we consider any two disjoint connected components of $\Omega$. Let they be $\cP$ and $\cQ$.
  Then we define,
  $$d(\cP, \cQ) = \inf \{ \dist(x,y): x \in \partial \cP , y \in \partial \cQ  \}. $$
  Similarly we can define,
  $$d(\cP_-, \cQ_-) = \inf \{ \dist(x,y): x \in \partial \cP_- , y \in \partial \cQ_-  \}, $$
  and
  $$d(\cP_+, \cQ_+) = \inf \{ \dist(x,y): x \in \partial \cP_+ , y \in \partial \cQ_+  \}. $$
  In all these definitions above, if any set, say $\cP_+$, is empty then the distance would be $+\infty$. 
  Now suppose there is a component $\cP$ of $\Omega$ such that $d(\cP, \cN_x) < r$ then one can see that $d(\cP_-, {\cN_x}_-)<r$. From the definition of infimum, there exit $a \in \partial {\cN_x}_-, b \in \cP_-$ such that $d(a,b)<r$. Since $\Omega$ has $C^{1,\beta}$ boundary, we must have $|B_r(a) \cap \cP_-|>0$. Now we already know that $a$ is an interior point and $\Omega = R(\Omega)$ everywhere in $B_r(a)$.  Hence $\Omega = R(\Omega)$ inside $B_r(a) \cap \cP_-$ and hence $b$ is an interior point. Now we extend the analysis on $b$ to the whole $\cP_-$ and get 
  $$\Omega = R(\Omega), \quad \hbox{inside}~ \cN_x \cup \cP. $$
  For a component $\cQ$, if $d(\cQ, \cN_x) \geq r$ then we can not extend analysis on $\cQ$ from $\cN_x$ directly. On the other hand, if $d(\cQ, \cP) <r$ then we can extend the analysis from $\cN_x$ to $\cQ$ via $\cP$. In that way, we say that $\cN_x$ can `influence' $\cQ$.  
  After doing analysis on each components, we get $\Omega$ to be union of sets of the following two types, which are at least at a distance $r$ from each other,
  \begin{enumerate}[(i)]
      \item A set, $S$ which is symmetric about $\pi_\lambda$. The  set $S$ may not be connected itself, rather it is a union of connected symmetric components of the form $ {\cN_x}$ described as before, where one component can `influence' other components.
      \item A set $\cPn$ which is not symmetric about $\pi_\lambda$. Again, this set $\cPn$ may not be connected itself, rather it is union of connected components which have no contact point. But again, one component can `influence' the other components.
  \end{enumerate}
  In the above construction, any such two sets (either symmetric or non-symmetric) have mutual distance at least  $r$. One can see that there can be only finitely many such sets as $\Omega$ is bounded. Now we denote the corresponding critical hyperplane $\pi_\lambda$ to be $\pi^e_{\lambda^{(1)}}$. With respect to this hyperplane, we get finitley many symmetric sets of the form $S$ and finitely many non-symmetric sets of the form $\cPn$. We write union of symmetric sets, in the first iteration as $\Omega_{s^{(e, 1)}}= \cup_{i =1}^{m_1} S^{(e,1)}_i$ and union of non-symmetric sets as $\Omega_{n^{(e,1)}} = \cup_{i =1}^{l_1} \cPn^{(e,1)}_i$. We denote $\Omega^{(e,1)} = \Omega = \Omega_{s^{(e,1)}} \cup \Omega_{n^{(e,1)}}$. In the next iteration, we only consider $\Omega_{n^{(e,1)}}$ and denote $\Omega^{(e,2)} = \Omega_{n^{(e,1)}}$. We already know each set in $\Omega^{e,2}$ does not have any contact point with respect to $\pi^e_{\lambda^{(1)}}$. Therefore we move the plane again until we get a new critical hyperplane $\pi^e_{\lambda^{(2)}}$. With respect to this hyperplane, we again have $\Omega^{(e,2)} =  \Omega_{s^{(e,2)}} \cup \Omega_{n^{(e,2)}}$. We continue this process for finitely many times as $\Omega$ is bounded and get $\Omega$ to be a symmetric in the direction of $e$ but for different critical hyperplanes. If the number of iteration or critical hyeprplane is $\Lambda^e$ then, 
  $$\Omega = \cup^{\Lambda^e}_{j=1} \cup^{m_j}_{i=1} S^{(e,j)}_i$$
  where $S^{(e,j)}_i$ is one of the $m_j$ number of symmetric sets (in spirit of definition given in (i)) with respect to the critical hyeprplane $\pi^e_{\lambda^{(j)}}$.
  
  Next we change direction to any $v \in \mathbb{S}^{n-1}$ such that $v \neq e$. We use the Alexandrov's principle and get a critical hyperplane in the first iteration for this vector $v$. Let us denote that hyperplane as $\pi^v_{\lambda^{(1)}}$. Consider a subset, $S^e$ of $\Omega$, which is symmetric with respect to the direction $e$ (in sense of definition (i)) and boundary of that set contains a point of contact with respect to $\pi^v_{\lambda^{(1)}}$.  Hence the component (of $S^e$) that contains the point of contact can be proved to be symmetric with respect to $\pi^v_{\lambda^{(1)}}$ using the earlier analysis.  Now remember that $S^e$ contains different connected components where one component can `influence' another component. Then it evident that $S^e$ is symmetric with respect to $\pi^v_{\lambda^{(1)}}$. Similarly all other sets of the form $S^{(e,j)}_i$ which are symmetric in the direction $e$, can be proved to be symmetric in the direction $v$ with respect to some critical hyperplane $\pi^v_{\lambda^{(k)}}$, where $k=1, 2, \cdots, \Lambda^v$ for some $0< \Lambda^v < \infty$. As $v$ is any vector, each set of the form $S^e$ is a ball. Then $\Omega$ is a union of balls of same radii, which are at least $r$ distance away from each other.
  \end{proof}

\section{Proof of Theorem \ref{thm:c2beta_bdd_3}}

As $J$ has jump at $\partial \Omega$ we can't directly use the computations as in the proof of Theorem \ref{thm:c2beta_bdd}. We can think $J(x) = \tilde{J}(x)\chi_{B_r}(x)$, where $\tilde{J}$ satisfies all the conditions of the kernel given in Theorem \ref{thm:c2beta_bdd}.
Let us define $\phi, \phi_\varepsilon : [0, \infty) \to \re_+$ exactly as in the proof of Theorem \ref{thm:c2beta_bdd} so that $\tilde{J}(x) = \phi(|x|)$ and $\phi_\varepsilon, \phi'_\varepsilon$ converges to $\phi, \phi'$ respectively. We also introduce another sequence of functions. First consider $\Tau:(-\infty, \infty)\to \re_{-}$ be such that $\Tau$ is $C^\infty_c(-\infty, \infty)$ with support in $(-r/2, r/2)$ and $\int_\R \Tau(t) dt = -1$. Now define the function
$g:[0, \infty) \to \re_+$ such that $g(t) = 1$ when $0\leq t \leq r$. Then it is evident that $\chi_{B_r}(x) = g(|x|)$. Next we define the sequence $g_\varepsilon: [0, \infty) \to \re_+$ such that $g_\varepsilon(0) = 1$ and $g'_\varepsilon(t) = \frac{1}{\varepsilon} \Tau \bigg(\frac{t-r}{\varepsilon} \bigg)$. Now we define, $\Lambda_\varepsilon:[0, \infty) \to \re_+ : t \mapsto \phi_\varepsilon(t) g_\varepsilon(t)$.

Before we introduce the main result of this section, we want to prove a supporting lemma. For this result to be true we want the boundary of $\Omega$ to be $C^1$ with uniformly bounded $C^1$-norm as mentioned in Theorem \ref{thm:c2beta_bdd_3}.

\begin{lem} \label{lem:C1differentiability}
Let $\Omega$ be a bounded set with $C^1$-boundary as mentioned in Theorem \ref{thm:c2beta_bdd_3}. 
Consider the following function,
$$H^c_\Omega(x) = \int_{\R^n} \chi_\Omega(y) \chi_{B_r(x)}(y) dy $$
for all $x \in \R^n$. Then $H^c_\Omega(x)$ has directional derivatives at every $x \in \R^n$. 
\end{lem}
\begin{proof}
First we define,
$$H^c_\varepsilon(x) = \int_{\Omega} g_\varepsilon(|x-y|)dy $$
then, for $e \in T_x(\partial \Omega)$, since $g_\varepsilon$ is $C^\infty$, we can write, 
$$\nabla H^c_\varepsilon(x) \cdot e = \int_{\Omega} \nabla g_\varepsilon(|x-y|) \cdot e dy. $$
Consider a point $x \in \partial \Omega$, $e$ is any tangent vector of $x$. Next we define,
$$F(x,e,\lambda) = \int_{\Omega \cap \partial B_\lambda(x)} (e \cdot e_y) d \sigma(y), $$
where $e_y$ is the unit normal vector at each $y \in \Omega \cap \partial B_\lambda(x)$. Then we have,
\begin{align*}
    \nabla H^c_\varepsilon(x) \cdot e &= \int_{\Omega} \nabla g_\varepsilon(|x-y|) \cdot e dy\\
    &= \int_{\Omega} g'_\varepsilon(|x-y|) \frac{(x-y)}{|x-y|} \cdot e dy \\
    &= \int^\infty_0 \int_{\Omega \cap \partial B_\lambda (x) } g'_\varepsilon(\lambda)(e_y \cdot e) d\sigma(y) d \lambda \quad [\hbox{where}~ x-y = \lambda e_y]\\
    &= \int^\infty_0 g'_\varepsilon(\lambda) F(x,e,\lambda) d \lambda \\
    &= \int^\infty_0 \frac{1}{\varepsilon} \Tau \bigg(\frac{\lambda -r}{\varepsilon}  \bigg) F(x,e,\lambda) d \lambda
\end{align*}
Now we want to prove that $F(x,e,\lambda)$ is a continuous function of $\lambda$ when $0 \leq \lambda < r+\delta$.
After rotation and translation, we can assume $x=0$, $\nabla f^x(0) =0$. For simplicity, we will now denote $f^x = f$. Next we define $B'_\rho = \{ y':|y'| <\rho \}$, $L_\rho = (-\rho, \rho) \times B'_\rho$,
$$\partial \Omega \cap L_{r+\delta} = \{(y',f(y')): y' \in B'_{r+ \delta} \} = (Id \times f) \big(B'_{r + \delta} \big) = G_{r+\delta},$$
and
$$\Omega \cap L_{r+\delta} = \{(y',y_n): y_n > f(y') \}. $$
 We have, for some $\lambda_1 < \lambda_2 < r+\delta$,
$$F(0,e,\lambda_2) - F(0,e,\lambda_1) = \int_{\Omega \cap \partial B_{\lambda_2}} (e \cdot e_y) d \sigma(y) - \int_{\Omega \cap \partial B_{\lambda_1}} (e \cdot e_y) d \sigma(y)  $$
If we define $h(y) = e \cdot y$ for any $y \in \R^n$ then $h(y)$ is a harmonic function and $\nabla_y h = e$. 
Next we observe that since $\norm{\nabla f} <1$, then for any tangent direction $\nu$ at the point $0$, $f$ intersects $\partial B_\lambda$ at only one point. That further implies 
$$\mathcal{H}^{n-1}(\partial \Omega \cap \partial B_\lambda(x)) =0 \quad \hbox{when}~\lambda < r+ \delta.$$
Next using the Gauss divergence theorem and proper orientation of the normal vector, we see that
$$\int_{\Omega \cap \partial B_{\lambda_2}} (e \cdot n_y) d \sigma(y) + \int_{\Omega \cap \partial B_{\lambda_1}} (e \cdot n_y) d \sigma(y) + \int_{\partial \Omega \cap ( B_{\lambda_2} \setminus B_{\lambda_1} )} (e \cdot n_y) d \sigma(y) = \int_{\omega} \Delta h dy =0, $$
where $\omega$ is the region bounded by the surfaces $\Omega \cap \partial B_{\lambda_2}, \Omega \cap \partial B_{\lambda_1}, \Omega \cap ( B_{\lambda_2} \setminus B_{\lambda_1} )$ and $n_y$ is the normal vector with the correct orientation. Therefore,
\begin{align*}
   F(0, e, \lambda_2) - F(0, e, \lambda_1) &= -\int_{\partial \Omega \cap ( B_{\lambda_2} \setminus B_{\lambda_1} )} (e \cdot n_y) d \sigma(y) \\
   &=-\int_{\partial \Omega \cap  B_{\lambda_2} } (e \cdot n_y) d \sigma(y) + \int_{\partial \Omega \cap  B_{\lambda_1} } (e \cdot n_y) d \sigma(y) . 
\end{align*}

Next we assume that $P_{\lambda}$ is the projection of $\partial \Omega \cap B_\lambda$ on the $x_n = 0$ hyperplane. Then $P_{\lambda_1} \subseteq P_{\lambda_2} \subseteq B'_{r+\delta}$. Using $n_y = \frac{(-\nabla f(y'), 1)}{\sqrt{1 + |\nabla f(y')|^2}}$ and $e \cdot e_n = 0$, we have $e \cdot n_y = \frac{-\nabla f(y')}{\sqrt{1+ |\nabla f(y')|^2}}$ and $d \sigma(y) =\sqrt{1+ |\nabla f(y')|^2} dy' $
\begin{align*}
   F(0, e, \lambda_2) - F(0, e, \lambda_1) &= -\int_{P_{\lambda_2}} (-\nabla f(y')) dy' + \int_{P_{\lambda_1}} (-\nabla f(y')) dy' \\
   =\int_{P_{\lambda_2}\setminus P_{\lambda_1}}\nabla f(y') dy'
\end{align*}

 That implies,
\begin{align*}
  \big|F(0,e,\lambda_2) - F(0,e, \lambda_1) \big| 
  &\leq \int_{P_{\lambda_2} \setminus P_{\lambda_1}} |\nabla f(y')| dy' \\
  &\leq \mathcal{H}^{n-1} \big( P_{\lambda_2} \setminus P_{\lambda_1} \big)= \mathcal{H}^{n-1}(P_{\lambda_2}) - \mathcal{H}^{n-1}(P_{\lambda_1})
\end{align*}
Next we fix a unit tangent vector $\nu$ of point $x=0$. We want to find $\mathcal{H}^1(P_{\lambda_2} \cap \nu ) - \mathcal{H}^1(P_{\lambda_1} \cap \nu )$. 

\begin{center}
    \hspace*{-30pt}
    \parbox{2.9in}{\scalebox{.71}{\includegraphics[trim=155pt 380pt 175pt 130pt,clip]{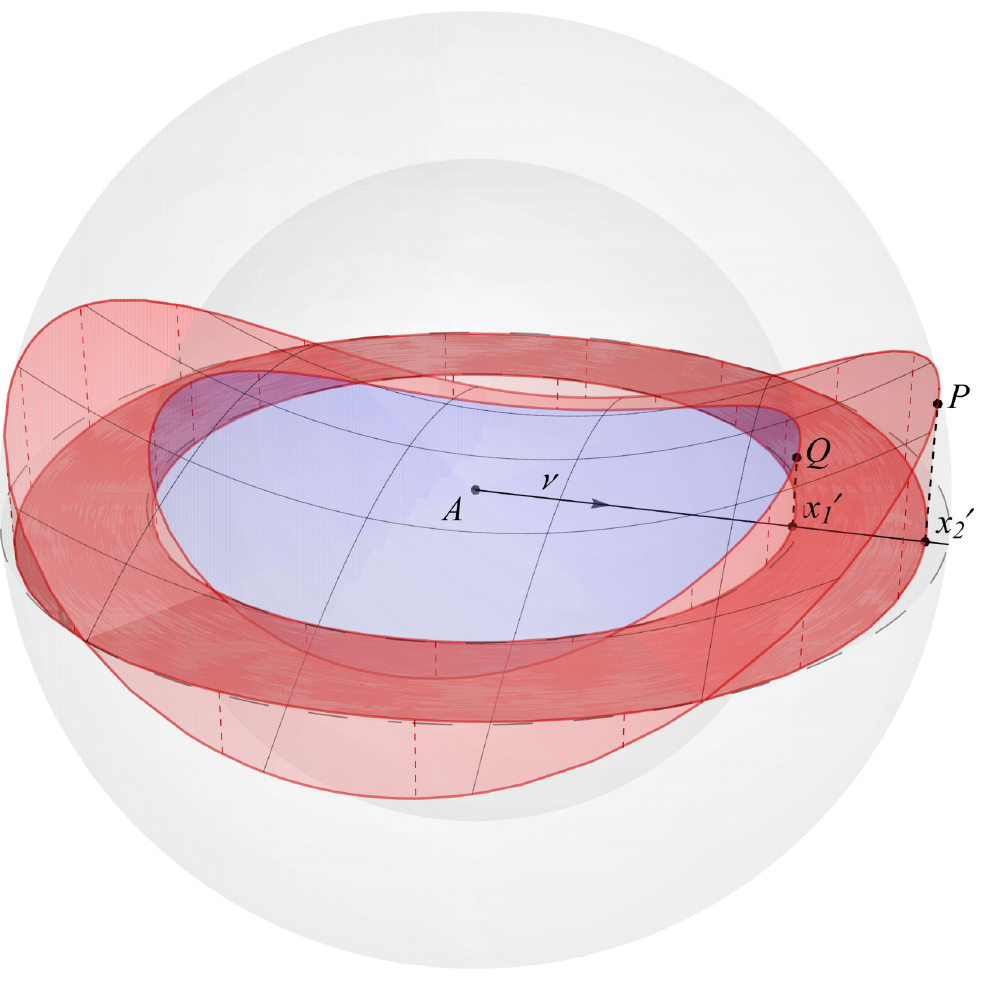}}}
    \parbox{2in}{\scalebox{.59}{\includegraphics[trim=155pt 330pt 125pt 130pt,clip]{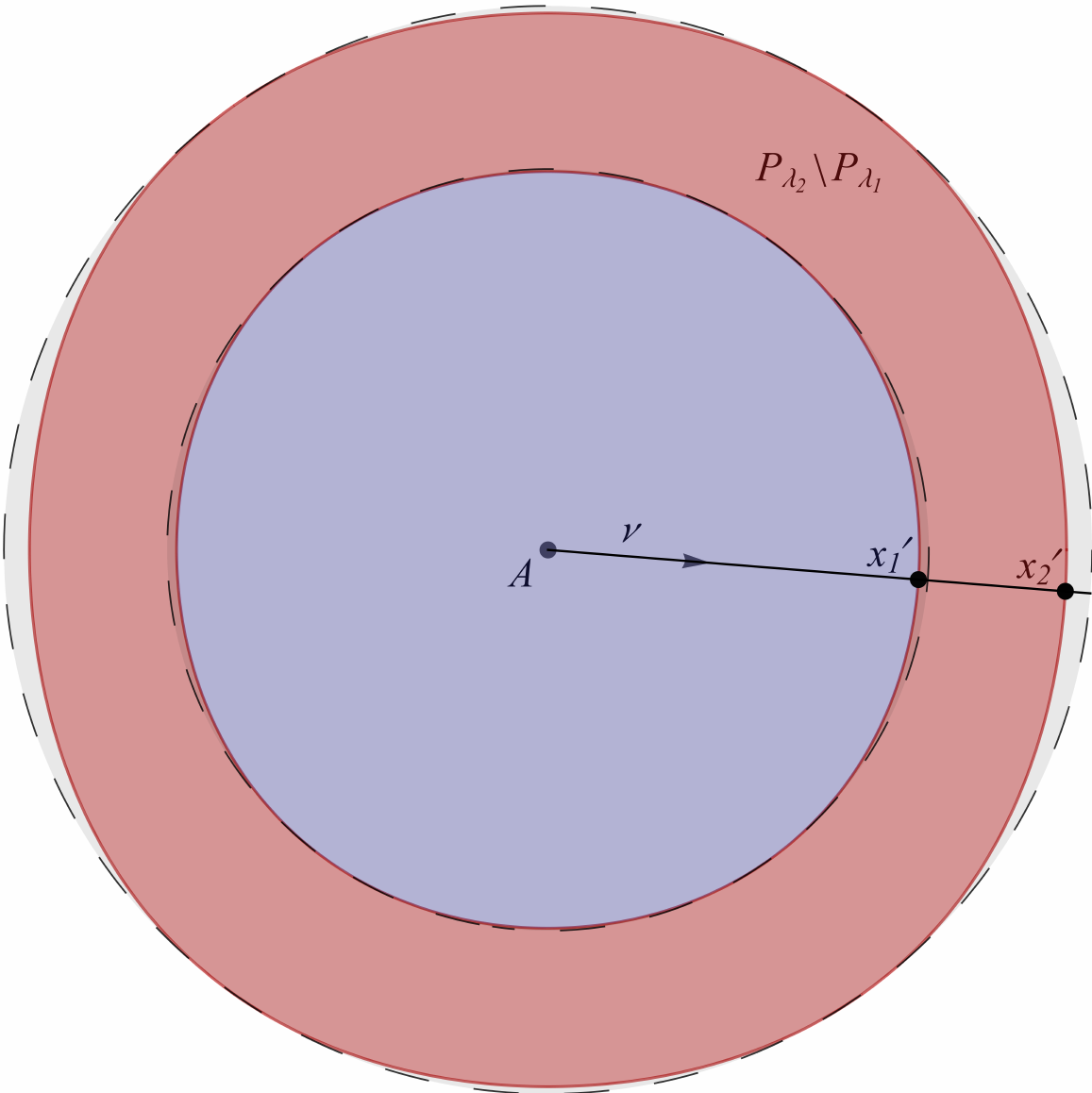}}}
\end{center}
\vspace{0.2 in}

Let $Q$ be the intersection point of $f$ with $\partial B_{\lambda_1}$ and $P$ be the intersection of $f$ with $\partial B_{\lambda_2}$ along the direction $\nu$. 
Next we define $G(x') = \sqrt{(x')^2 + f(x')^2}$ and along that tangent direction $\nu$, we use $\rho$ to denote the distance (from $0$) of a point in $x_n=0$  hyperplane. Then $x' = \rho \nu$ and $Q$ corresponds to $G(x'_1) = \lambda_1$, $P$ corresponds to $G(x'_2) = \lambda_2$. Or we can also say that
$$G(\rho_1;\nu) = \sqrt{\rho^2_1 + f(\rho_1 \nu)^2} = \lambda_1, \quad G(\rho_2;\nu) =\sqrt{\rho^2_2 + f(\rho_2 \nu)^2} = \lambda_2 $$
Since for a fixed $\nu$, there is a unique $\rho$ for each $\lambda$ and hence we we can write $\rho = Gn(\lambda;\nu)$. Then 
$$G(Gn(\lambda;\nu);\nu) = \lambda. $$
Next differentiating w.r.t $\lambda$,
$$\frac{dG}{d\rho} \frac{d \rho}{d \lambda} = 1. $$
Now $$\frac{dG}{d \rho} = \nabla_{x'} G \cdot \frac{x'}{\rho}$$
where $\nabla_{x'} G = \frac{x' + f(x') \nabla f(x')}{G(x')}$. Using the fact that $|\nabla f(x')| \leq M$, we have $|f(x') \nabla f(x') \cdot x'| \leq M^2 |x'|^2$. That further implies $\frac{dG}{d \rho} \geq \frac{(1-M^2)|x'|^2}{|x'| G(x')} = \frac{(1-M^2)\rho}{G(x')}= \frac{(1-M^2)\rho}{\lambda}$. Then we have
$$ \frac{d \rho}{d \lambda} \leq \frac{\lambda}{(1-M^2) \rho}$$
Next, we notice that $\rho \geq \frac{\lambda}{\sqrt{1+M^2}}$ which implies 
$$\frac{d \rho}{d \lambda} \leq \frac{\sqrt{1+M^2}}{(1-M^2) }. $$
Hence we have, if $(\lambda_2 - \lambda_1)$ is small, 
$$\mathcal{H}^1(P_{\lambda_2} \cap \nu) - \mathcal{H}^1(P_{\lambda_2} \cap \nu) =  Gn(\lambda_2;\nu) - Gn(\lambda_1;\nu)\leq C(\lambda_2 - \lambda_1)$$
This is true for any tangent vector $\nu$ which further implies that the distance between the $\partial P_{\lambda_2}$ and $\partial P_{\lambda_1}$ is always bounded by $C(\lambda_2 - \lambda_1)$ in any direction, when $(\lambda_2 - \lambda_1)$ is small. Furthermore, we have 
\begin{align*}
\mathcal{H}^{n-1}(P_{\lambda_2}) - \mathcal{H}^{n-1}(P_{\lambda_2}) &=  \int_{B^{n-2}} \int^{Gn(\lambda_2;\nu)}_{Gn(\lambda_1;\nu)} \rho^{n-2} d\rho \, d \Theta(\nu) \\
&= \int_{B^{n-2}} \frac{1}{n-1} \big[Gn(\lambda_2;\nu)^{n-1} - Gn(\lambda_1;\nu)^{n-1}]\, d \Theta(\nu) \\
&\leq C_n (\lambda_2 - \lambda_1)
\end{align*}

As a consequence, $F(0,e,\lambda)$ is continuous in $\lambda$ and the modulus of continuity does not depend on $x,e$.  Then we can write,
\begin{align*}
  &\bigg|\int^\infty_0 \frac{1}{\varepsilon} \Tau \bigg(\frac{\lambda - r}{\varepsilon} \bigg) F(x,e,\lambda) \, d\lambda +  F(x,e,r) \bigg| \\
  &=\bigg|\int^\infty_0  \frac{1}{\varepsilon} \Tau \bigg(\frac{\lambda - r}{\varepsilon} \bigg) F(x,e,\lambda) \, d\lambda - \int^\infty_0 \frac{1}{\varepsilon} \Tau \bigg(\frac{\lambda - r}{\varepsilon} \bigg)  F(x,e,r) d \lambda \bigg|\\
  &= \bigg| \int^\infty_0 \frac{1}{\varepsilon} \Tau \bigg(\frac{\lambda - r}{\varepsilon} \bigg) \big(   F(x,e,\lambda) - F(x,e,r) \big) \, d\lambda  \bigg| \\
  &\leq  \int_{|\lambda -r|\leq r \varepsilon/2} \bigg| \frac{1}{\varepsilon} \Tau \bigg(\frac{\lambda - r}{\varepsilon} \bigg) \big(   F(x,e,\lambda) - F(x,e,r) \big) \bigg| \, d\lambda \\
  &\quad + \int_{|\lambda -r|> r \varepsilon/2} \bigg| \frac{1}{\varepsilon} \Tau \bigg(\frac{\lambda - r}{\varepsilon} \bigg) \big(  F(x,e,\lambda) - F(x,e,r) \big) \bigg| \, d\lambda
\end{align*}
Since $ F(x,e,\lambda)$ is uniformly continuous in the interval $[0, r+\delta/2]$ then for small $\varepsilon_0$, as $\varepsilon \to 0$, 
$$\int^\infty_0  \frac{1}{\varepsilon} \Tau \bigg(\frac{\lambda - r}{\varepsilon} \bigg) F(x,e,\lambda) \, d\lambda \to  F(x,e,r) = - \int_{\Omega \cap \partial B_r(x)} (e_y \cdot e )\, d\sigma(y). $$  
Then $H^c_\varepsilon$ converges in $C^1(\partial \Omega)$ norm. 
\end{proof}

\begin{rem}
Without the given boundary conditions, as described in Theorem \ref{thm:c2beta_bdd_3}, curvature function may not be differentiable. Let us consider the following example. Suppose $0 \in \partial \Omega$ and the boundary function $f$ follows some part of the sphere $\partial B_r$. Then we can see that $F$ is not continuous at the point $\lambda = r$. 
\end{rem}

\begin{lem}\label{lem:C1differentiability_2}
Let $\Omega$ be a bounded set with $C^1$-boundary as mentioned in Theorem \ref{thm:c2beta_bdd_3}. 
Consider the following function,
$$H^c_\Omega(x) = \int_{\R^n} \tau_\Omega(y) \chi_{B_r(x)}(y) dy $$
for all $x \in \R^n$. Then $H^c_\Omega$ is in $C^1(\partial \Omega)$. 
\end{lem}

\begin{proof}
This proof is a consequence of Lemma \ref{lem:C1differentiability} after defining 
$$F(x,e,\lambda) = \int_{\Omega^c \cap \partial B_\lambda(x)} (e \cdot e_y) d \sigma(y) - \int_{\Omega \cap \partial B_\lambda(x)} (e \cdot e_y) d \sigma(y), $$
where $e_y$ is the unit normal vector at each $y \in \Omega \cap \partial B_\lambda(x)$. 

\end{proof}


\begin{lem}\label{lem:Figalli 2.1 with jump}
We consider $\Lambda_\varepsilon$ which is defined as above. Then we define 
$$H^\varepsilon_\Omega (x) = \int_{\re^n} \tau_\Omega (y) \Lambda_\varepsilon(|x-y|) dy,  $$
$x \in \re^n$. Next we assume that $\Omega$ is nonempty bounded open set with $C^1$-boundary as mentioned in Theorem \ref{thm:c2beta_bdd_3}. In addition to that, $\partial \Omega$  is $C^{1, \beta}$ for some $\beta =\min \{\alpha - 1, 0 \}$.  Then $H^J_\Omega \in C^1(\partial \Omega)$ and $H^\varepsilon_\Omega \to H^J_\Omega$ in $C^1(\partial \Omega)$ as $\varepsilon \to 0$.
\end{lem}
\begin{proof}
Using the definition of $H^\varepsilon_\Omega$, we compute, for some $e \in T_x(\partial \Omega)$,
$$\nabla H^\varepsilon_\Omega(x) \cdot e = \int_{\re^n}\tau_\Omega(y) {\phi_\varepsilon}(|x-y|) \nabla g_\varepsilon(|x-y|) \cdot e \, dy + \int_{\re^n}\tau_\Omega(y) (\nabla {\phi_\varepsilon}(|x-y|) \cdot e)  g_\varepsilon(|x-y|)  \, dy $$
We notice that 
$$\int_{\re^n}\tau_\Omega(y) (\nabla {\phi_\varepsilon}(|x-y|) \cdot e) dy \to \int_{\re^n}\tau_\Omega(y) (\nabla {\phi}(|x-y|) \cdot e) dy $$
uniformly in $C^0(\partial \Omega)$ using Lemma \ref{lem:Figalli 2.1}. Then we have, 
\begin{align*}
    &\bigg| \int_{\re^n}\tau_\Omega(y) (\nabla {\phi_\varepsilon}(|x-y|) \cdot e) dy - \int_{\re^n}\tau_\Omega(y) (\nabla {\phi_\eta}(|x-y|) \cdot e) dy\bigg| \\
    &= \bigg| \int_{\re^n}\tau_\Omega(y) [(\nabla {\phi_\varepsilon}(|x-y|) - \nabla {\phi_\eta}(|x-y|)) \cdot e] dy\bigg|\\
    &= \bigg| \int_{\re^n}\tau_\Omega(z+x) [(\nabla {\phi_\varepsilon}(|z|) - \nabla {\phi_\eta}(|z|)) \cdot e] dz\bigg|\\
    &= \bigg| \int_{\re^n}\tau_\Omega(z+x) (\phi_\varepsilon'(|z|) - \phi'_\eta(|z|)) \bigg(\frac{z \cdot e}{|z|}\bigg) dz\bigg| \\
    &= \bigg| \int_{B_\eta}\tau_\Omega(z+x) (\phi_\varepsilon'(|z|) - \phi'_\eta(|z|)) \bigg(\frac{z \cdot e}{|z|}\bigg) dz\bigg| \to 0 , \quad \hbox{when}~ \varepsilon, \eta \to 0. 
\end{align*}
In the last line, $B_\eta$ denotes the ball of radius $\eta$ centered at $0$ and we assume $\eta < \varepsilon$. Then we use the fact that $\phi_\varepsilon(t) = \phi_\eta(t)$ when $t >\varepsilon$. Next from the uniform convergence of  we have 
$$\bigg| \int_{B_\eta}\tau_\Omega(z+x) (\phi_\varepsilon'(|z|) - \phi'_\eta(|z|)) \bigg(\frac{z \cdot e}{|z|}\bigg) dz\bigg| \to 0$$
Next we define $A_{\varepsilon} =\{ z : 0< g_\varepsilon(|z|) <1 \} = B_{r(1+\eta/2)} \setminus B_{r(1-\eta/2)}$. Since $g_\varepsilon \to g$, then $|A_\varepsilon| \to 0$. 
Then we have, for very small $\varepsilon, \eta >0$ and $\eta > \varepsilon$,
\begin{align*}
    &\bigg|\int_{\re^n}\tau_\Omega(y) (\nabla {\phi_\varepsilon}(|x-y|) \cdot e)  g_\varepsilon(|x-y|)  \, dy - \int_{\re^n}\tau_\Omega(y) (\nabla {\phi_\eta}(|x-y|) \cdot e)  g_\eta(|x-y|)  \, dy \bigg|\\
    &= \bigg|\int_{\re^n}\tau_\Omega(x+z)(\phi_\varepsilon'(|z|)g_\varepsilon(|z|) - \phi'_\eta(|z|) g_\eta(|z|)) \bigg(\frac{z \cdot e}{|z|}\bigg) dz  \bigg|\\
    &\leq \bigg|\int_{\re^n \setminus (B_\eta \cup A_\eta)}\tau_\Omega(x+z)(\phi_\varepsilon'(|z|)g_\varepsilon(|z|) - \phi'_\eta(|z|) g_\eta(|z|)) \bigg(\frac{z \cdot e}{|z|}\bigg) dz  \bigg| \\
    &\qquad + \bigg|\int_{B_\eta}\tau_\Omega(x+z)(\phi_\varepsilon'(|z|)g_\varepsilon(|z|) - \phi'_\eta(|z|) g_\eta(|z|)) \bigg(\frac{z \cdot e}{|z|}\bigg) dz  \bigg| \\
    &\qquad + \bigg|\int_{A_\eta}\tau_\Omega(x+z)(\phi_\varepsilon'(|z|)g_\varepsilon(|z|) - \phi'_\eta(|z|) g_\eta(|z|)) \bigg(\frac{z \cdot e}{|z|}\bigg) dz  \bigg|
\end{align*}
From the construction of $\phi_\varepsilon$ and $g_\varepsilon$, we notice that $\phi_\varepsilon = \phi_\eta$ in $(\eta, \infty)$ and $g_\varepsilon = g_\eta$ in $[0, r(1-\eta/2)) \cup (r(1+\eta/2), \infty)$. Since $\varepsilon, \eta$ are very small,, $B_\eta \cap A_\eta = \emptyset$. Using that we get,
\begin{align*}
    &\bigg|\int_{\re^n}\tau_\Omega(y) (\nabla {\phi_\varepsilon}(|x-y|) \cdot e)  g_\varepsilon(|x-y|)  \, dy - \int_{\re^n}\tau_\Omega(y) (\nabla {\phi_\eta}(|x-y|) \cdot e)  g_\eta(|x-y|)  \, dy \bigg|\\
    &\leq \bigg|\int_{B_\eta}\tau_\Omega(x+z)(\phi_\varepsilon'(|z|) - \phi'_\eta(|z|)) \bigg(\frac{z \cdot e}{|z|}\bigg) dz  \bigg| \\
    &\qquad + \bigg|\int_{A_\eta}\tau_\Omega(x+z)(g_\varepsilon(|z|) - g_\eta(|z|)) \bigg(\frac{z \cdot e}{|z|}\bigg) dz  \bigg|
\end{align*}
We see that the first term converges to $0$ as $\eta, \varepsilon \to 0$. On the other hand, the second term is bounded by $4|A_\eta|$ which converges to $0$ when $\eta \to 0$. Hence we have 
$$\int_{\re^n}\tau_\Omega(y) (\nabla {\phi_\varepsilon}(|x-y|) \cdot e)  g_\varepsilon(|x-y|)  \, dy \to \int_{\re^n}\tau_\Omega(y) (\nabla {\phi}(|x-y|) \cdot e)  g(|x-y|)  \, dy $$
uniformly for every $x \in \partial \Omega$.

 Next we consider the following term,
 $$ \int_{\re^n}\tau_\Omega(y) {\phi_\varepsilon}(|x-y|) \nabla g_\varepsilon(|x-y|) \cdot e \, dy .$$
  We can write 
\begin{align*}
    & \int_{\re^n}\tau_\Omega(y) \phi_\varepsilon(|x-y|) \nabla  g_\varepsilon(|x-y|) \cdot e \, dy \\
    &= \int_{\Omega^c} \phi_\varepsilon(|x-y|) g'_\varepsilon(|x-y|) \frac{(x-y)}{|x-y|}\cdot e \, dy \\
    &\quad - \int_{\Omega} \phi_\varepsilon(|x-y|) g'_\varepsilon(|x-y|) \frac{(x-y)}{|x-y|}\cdot e \, dy \\
    &= \int^\infty_0 \int_{\Omega^c \cap \partial B_\gamma(x)}\phi_\varepsilon(\gamma) g'_\varepsilon(\gamma) (e_y \cdot e )\, d\sigma(y) d\gamma \\
    &\quad - \int^\infty_0 \int_{\Omega \cap \partial B_\gamma(x)}\phi_\varepsilon(\gamma) g'_\varepsilon(\gamma)( e_y \cdot e )\, d\sigma(y) d\gamma, \quad [\text{where}~ (x-y) = \gamma e_y]
\end{align*}
If we define
$$F(x,e,\lambda) = \int_{\Omega^c \cap \partial B_\lambda(x)} (e \cdot e_y) d \sigma(y) - \int_{\Omega \cap \partial B_\lambda(x)} (e \cdot e_y) d \sigma(y), $$
as before, then we notice that $F(x,e,\lambda)$ is continuous function in $[0, r+\delta/2]$.
Next, using the explicit expression of $g'_\varepsilon$ we get,
$$ \int^\infty_0 \phi_\varepsilon(\gamma) g'_\varepsilon(\gamma) F(x,e, \gamma) \, d\gamma =  \int^\infty_0 \phi_\varepsilon(\gamma) \frac{1}{\varepsilon} \Tau \bigg(\frac{\gamma - r}{\varepsilon} \bigg) F(x,e,\gamma) \, d\gamma$$
Then
\begin{align*}
  &\bigg|\int^\infty_0 \phi_\varepsilon(\gamma) \frac{1}{\varepsilon} \Tau \bigg(\frac{\gamma - r}{\varepsilon} \bigg) F(x,e,\gamma) \, d\gamma + \phi(r) F(x,e,r) \bigg| \\
  &=\bigg|\int^\infty_0 \phi_\varepsilon(\gamma) \frac{1}{\varepsilon} \Tau \bigg(\frac{\gamma - r}{\varepsilon} \bigg) F(x,e,\gamma) \, d\gamma - \int^\infty_0 \frac{1}{\varepsilon} \Tau \bigg(\frac{\gamma - r}{\varepsilon} \bigg) \phi(r) F(x,e,r) d \gamma \bigg|\\
  &= \bigg| \int^\infty_0 \frac{1}{\varepsilon} \Tau \bigg(\frac{\gamma - r}{\varepsilon} \bigg) \big( \phi_\varepsilon(\gamma)  F(x,e,\gamma) -\phi(r) F(x,e,r) \big) \, d\gamma  \bigg| \\
  &\leq  \int_{|\gamma -r|\leq r \varepsilon/2} \bigg| \frac{1}{\varepsilon} \Tau \bigg(\frac{\gamma - r}{\varepsilon} \bigg) \big( \phi_\varepsilon(\gamma)  F(x,e,\gamma) -\phi(r) F(x,e,r) \big) \bigg| \, d\gamma \\
  &\quad + \int_{|\gamma -r|> r \varepsilon/2} \bigg| \frac{1}{\varepsilon} \Tau \bigg(\frac{\gamma - r}{\varepsilon} \bigg) \big( \phi_\varepsilon(\gamma)  F(x,e,\gamma) -\phi(r) F(x,e,r) \big) \bigg| \, d\gamma
\end{align*}
Using the continuity of $F$, as $\varepsilon \to 0$, 
\begin{align*}
   & \int_{\re^n}\tau_\Omega(y) \phi_\varepsilon(|x-y|) \nabla  g_\varepsilon(|x-y|) \cdot e \, dy \\
    &  \xrightarrow[\varepsilon \to 0]{}   -\phi(r) \bigg[ \int_{\Omega^c \cap \partial B_r(x)} (e_y \cdot e )\, d\sigma(x,y) -   \int_{\Omega \cap \partial B_r(x)} (e_y \cdot e )\, d\sigma(y) \bigg] \\
    &\quad = -\phi(r)\int_{\partial B_r(x)} \tau_\Omega(y)  (e_y \cdot e )\, d\sigma(x,y)
\end{align*}
Then $H^\varepsilon_\Omega$ converges in $C^1(\partial \Omega)$ norm. 
\end{proof}
Now we will prove the Theorem \ref{thm:c2beta_bdd_3}. Again we move hyperplane and do the reflection for a vector $e \in \mathbb{S}^{n-1}$. Suppose $\pi_\lambda$ is the critical hyperplane in this case and for simplicity $e = \langle 1, 0, \cdots, 0 \rangle$. 
\begin{proof}[Proof of Theorem \ref{thm:c2beta_bdd_3}]
As before, we consider two different cases. 
\begin{enumerate}[(a)]
    \item Interior touching at the point $x_0$. In this case, if we follow similar computation as in the proof of Theorem $\ref{thm:c2beta_bdd}$, we get 
    $$|(\Omega \setminus R(\Omega))\cap B_r(x_0)|= |(R(\Omega) \setminus \Omega) \cap B_r(x_0)| = 0 $$
    \item Non-transversal intersection at the point $x_0$. Again we follow similar type of computation as in the proof of Theorem \ref{thm:c2beta_bdd}. That gives us $\partial_e H^J_\Omega (x_0) - \partial_e H^J_{R(\Omega)}(x_0) = 0$. Now we know that 
    $$ \partial_e H^J_\Omega(x_0) = \lim_{\varepsilon \to 0} \int_{\re^n} \tau_\Omega \phi'_\varepsilon(|x_0-y|)\frac{(x_0)_1 - y_1}{|x_0 - y|} g_\varepsilon(|x_0-y|) dy - \phi(r)\int_{\partial B_r(x)} \tau_\Omega(y)  \frac{(x_0)_1 - y_1}{|x_0-y|}\, d\sigma(x,y)$$
    Then
    \begin{align*}
        &0=\partial_e H^J_\Omega (x_0) - \partial_e H^J_{R(\Omega)}(x_0) \\
        &= \lim_{\varepsilon \to 0} \int_{\re^n} (\tau_\Omega - \tau_{R(\Omega)}) \phi'_\varepsilon(|x_0-y|)\frac{(x_0)_1 - y_1}{|x_0 - y|} g_\varepsilon(|x_0-y|) dy \\
        &\quad - \phi(r)\int_{\partial B_r(x)} (\tau_\Omega(y)-\tau_{R(\Omega)}(y))  \frac{(x_0)_1 - y_1}{|x_0-y|}\, d\sigma(x,y)\\
        &= \lim_{\varepsilon \to 0} \bigg[ \int_{\Omega \setminus R(\Omega)} \phi'_\varepsilon(|x_0-y|)g_\varepsilon(|x_0-y|) \frac{(x_0)_1-y_1}{|x_0-y|} \, dy \\
        &\qquad - \int_{R(\Omega) \setminus \Omega} \phi'_\varepsilon(|x_0-y|)g_\varepsilon(|x_0-y|) \frac{(x_0)_1-y_1}{|x_0-y|} \, dy \bigg] \\
        &\qquad - \phi(r) \bigg[ \int_{(\Omega \setminus R(\Omega)) \cap \partial B_r(x)} \frac{(x_0)_1-y_1}{|x_0-y|} \, dy - \int_{(R(\Omega) \setminus \Omega)\cap \partial B_r(x)}\frac{(x_0)_1-y_1}{|x_0-y|} \, dy \bigg]
    \end{align*}
    Next we observe that $(x_0)_1 = \lambda$ and if $y\in \Omega \setminus R(\Omega)$ then $y_1<\lambda$. On the other hand, if $y \in R(\Omega) \setminus \Omega$ then $y_1> \lambda$. Hence
    \begin{align*}
      &\lim_{\varepsilon \to 0} \bigg[ \int_{\Omega \setminus R(\Omega)} \phi'_\varepsilon(|x_0-y|)g_\varepsilon(|x_0-y|) \frac{|\lambda- y_1|}{|x_0-y|} \, dy + \int_{R(\Omega) \setminus \Omega} \phi'_\varepsilon(|x_0-y|)g_\varepsilon(|x_0-y|) \frac{|\lambda -y_1|}{|x_0-y|} \, dy \bigg] \\
        &\quad - \phi(r) \bigg[ \int_{(\Omega \setminus R(\Omega)) \cap \partial B_r(x)} \frac{|\lambda -y_1|}{|x_0-y|} \, dy + \int_{(R(\Omega) \setminus \Omega)\cap \partial B_r(x)}\frac{|\lambda -y_1|}{|x_0-y|} \, dy \bigg] = 0 
    \end{align*}
    We know $g_\varepsilon \geq 0$ in $B_r$, $\phi(r)>0, \phi'_\varepsilon \leq 0$. 
    From this we can say that 
    $$|(\Omega \setminus R(\Omega))\cap B_r(x_0)|= |(R(\Omega) \setminus \Omega) \cap B_r(x_0)| = 0 $$
    and
    $$\mathcal{H}^{n-1} ((\Omega \setminus R(\Omega)) \cap \partial B_r(x)) = \mathcal{H}^{n-1} ((R(\Omega) \setminus \Omega)\cap \partial B_r(x)) = 0$$
 \end{enumerate}
 Hence in both cases, we are getting that,
 $$|(\Omega \setminus R(\Omega))\cap B_r(x_0)|= |(R(\Omega) \setminus \Omega) \cap B_r(x_0)| = 0 $$
 Now using the arguments as in the proof of Theorem \ref{thm:c2beta_bdd_2}, we get the desired result.
\end{proof}

\section{Constant Curvature problem with constant kernel}
For this case, the prototypical kernel is 
$$J(x) = \chi_{B_r}(x), $$
where $B_r$ is the ball of radius $r$ centered at $0$.
Then we see that the nonlocal curvature at a point $x$ is
\begin{align*}
    H^J_\Omega (x) &= \int_{\R^n} (\chi_{\Omega^c}(y) - \chi_{\Omega}(y))J(x-y) dy \\
    &= \int_{\R^n} (\chi_{\Omega^c}(y) - \chi_{\Omega}(y))\chi_{B_r(x)}(y) dy \\
    &= |\Omega^c \cap B_r(x)| - |\Omega \cap B_r(x)| \\
    &= |B_r| - 2 |\Omega \cap B_r(x)|
\end{align*}
As before, we assume $e=\langle 1, 0, \cdots, 0 \rangle$ with $\pi_\lambda$ to be the critical hyperplane for the direction $e$. $R_\lambda$ is the reflection operator with respect to $\pi_\lambda$, which will be denoted as operator $R$ for simplicity. Here again, we define,
$${\pi_\lambda}_+ = \{x: x_1 > \lambda \} \quad \text{and}~ \quad  {\pi_\lambda}_- = \{x: x_1 < \lambda \}$$
We also define
$$
\begin{cases}
\Omega_+ = \Omega \cap {\pi_\lambda}_+, \quad \text{and}~\quad \Omega_- = \Omega \cap {\pi_\lambda}_- \\
\partial \Omega_+ = \partial \Omega \cap {\pi_\lambda}_+, \quad \text{and}~\quad \partial \Omega_- = \partial \Omega \cap {\pi_\lambda}_-
\end{cases}
$$

\begin{proof}[Proof of Theorem \ref{thm:thm:c2beta_bdd_4}]
As before, we fix the direction $e$ and apply the Alexandrov's moving plane method. Next we consider two different cases.

 Case (a): Interior touching at the point $x_0$. In this case, if we follow similar computation as in the proof of Theorem $\ref{thm:c2beta_bdd}$, we get,
 \begin{align}\label{eq:interior constant kernel}
        0 &= H^J_\Omega(x_0) - H^J_\Omega(R(x_0)) \\ \nonumber
        &= H^J_\Omega(x_0) - H^J_{R(\Omega)}(x_0)\\ \nonumber
        &= \int_{\re^n} \bigg( \tau_{\Omega}(y) - \tau_{R(\Omega)}(y)\bigg)J(x_0 - y) dy\\ \nonumber
        &= \int_{\Omega \setminus R(\Omega)} J(x_0-y) dy - \int_{R(\Omega) \setminus \Omega} J(x_0 -y) dy \\ \nonumber
        &= \int_{\Omega \setminus R(\Omega)} \big(J(x_0-y)-J(x_0 - R(y)) dy \\ \nonumber
        &= |(\Omega \setminus R(\Omega))\cap (B_r(x_0) \setminus B_r(R(x_0))|
    \end{align}
 $\Omega$ is a connected set with $C^1$ boundary and hence for any $x_1, x_2 \in \partial \Omega$, if $x_1 \neq x_2$, we have
 $$|\Omega \cap (B_r(x_1) \Delta B_r(x_2))| >0. $$ 
 Hence from \eqref{eq:interior constant kernel}, we have $\Omega = R(\Omega)$ inside the set $B_r(x_0) \setminus B_r(R(x_0))$ or the boundaries of $\Omega$ and $R(\Omega)$ coincide inside the set $B_r(x_0) \setminus B_r(R(x_0))$. In fact, boundaries of $\Omega$ and $R(\Omega)$ coincide inside the set $B_r(x_0) \setminus \overline{B_r(R(x_0))}$. Let $y \in B_r(x_0) \setminus \overline{B_r(R(x_0))}$ be such a point on $\partial \Omega_-$. Then $y$ is a point of interior touching.
Since $B_r(x_0) \setminus \overline{B_r(R(x_0))}$ is open, there exists an open neighborhood $V \subseteq B_r(x_0) \setminus \overline{B_r(R(x_0))}$ such that on the set, $V \cap \partial \Omega_-$, $\Omega$ and $R(\Omega)$ coincide and $y \in V \cap \partial \Omega_-$.  Since $y \in B_r(x_0) \setminus \overline{B_r(R(x_0))}$,  it is evident that $x_0 \in B_r(y) \setminus \overline{B_r(R(y))}$ and similarly there exits an open neighborhood $U \subseteq B_r(y) \setminus \overline{B_r(R(y))}$ such that $\Omega = R(\Omega)$ inside $U$ and $x_0 \in U \cap \partial \Omega_-$.
 Let $\partial^i \Omega$ be the set of all interior points. Then, it is an open set subset of $\partial \Omega_-$. Next consider a point $x_{nc} \in \partial \Omega_- \setminus \partial^i \Omega$. That point is not a point of contact, i.e neither an interior point nor a point of non-transversal intersection. Next we consider the set $\partial \Omega_- \cap (B_r(x_{nc}) \setminus \overline{B_r(R(x_{nc}))})$. Every point in that set would be a noncontact point. Otherwise, if there is one contact point, say $z_c$, that would be an interior point and $x_{nc} \in \partial \Omega_- \cap (B_r(z_c) \setminus \overline{B_R(R(z_c))})$, which makes $x_{nc}$ to be an interior point. That is a contradiction. Next, using the similar argument as before, we see that the set of non-contact point of $\partial \Omega_-$ is also an open subset of $\partial \Omega_-$. But that is not possible since $\partial \Omega_-$ is a connected set. That implies $\partial^i \Omega = \partial \Omega_-$ and hence all points on $\partial \Omega_-$ are interior touching points. 
 
 Case (b): Non-transversal intersection at the point $x_0$. Using Lemma \ref{lem:Figalli 2.1 with jump}, we see that 
$$\partial_e H^J_\Omega(x_0) = - \int_{\partial B_r(x)} \tau_\Omega(y)  \frac{(x_0)_1 - y_1}{|x_0-y|}\, d\sigma(x,y) $$
Similar analysis as in the proof of Theorem \ref{thm:c2beta_bdd_3}, provides,
\begin{align}\label{eq:constant_non_transverse}
    \mathcal{H}^{n-1} ((\Omega \setminus R(\Omega)) \cap \partial B_r(x_0)) = \mathcal{H}^{n-1} ((R(\Omega) \setminus \Omega)\cap \partial B_r(x_0)) = 0.
\end{align}
Since $\Omega$ has $C^1$-boundary, for every point $x_0 \in \partial \Omega$, 
$$\mathcal{H}^{n-1} (\Omega \cap \partial B_r(x_0)) >0. $$
In particular, 
\begin{align*}
    \mathcal{H}^{n-1} (\Omega_- \cap \partial B_r(x_0)) >0.
\end{align*}
Then \eqref{eq:constant_non_transverse} implies $\Omega = R(\Omega)$ on $\partial B_r(x_0) \cap \Omega_-$. From the given property of the boundary, there will be a point $z \in \partial \Omega_- \cap \partial B_r(x_0)$. But then that point would be an interior touching point. Next we do the analysis as in Case (a) to conclude all the points on $\partial \Omega_-$ are interior touching points.

Then in both cases, we proved that $\Omega$ is symmetric with respect to the direction $e$. Since $e$ is arbitrary, $\Omega$ is a ball.
\end{proof}

\section{Conclusions and future directions}

The results of this paper showcases the extent to which previous methods (such as the Alexandrov's moving plane, establishing differentiability of the nonlocal curvature) can handle problems of constant nonlocal mean curvature, while allowing kernels which are only integrable. In fact, the kernel's integrability becomes an advantage as it enables our analysis to eliminate some of the smoothness restrictions on the boundary, which were present in previous works; indeed, we are able to consider boundaries with $C^1$ regularity whenever the kernel is constant. So far it remains unclear if continuous boundaries could be handled with a modified version of this approach. For this to work, one would need a generalized normal vector, for which the authors are investigating some possible candidates.

In conclusion, the main problem arising for future studies is considering sets with little regularity of the boundary. Note that the nonlocal curvature function is not even $C^1$ in the absence of the same regularity for the boundary. We are also interested in the nonconstant prescribed curvature problem. More precisely, for the prescribed curvature problem, we aim to establish wellposedness and regularity results for a function $u$, whose graph $\partial E$ satisfies the equation  $H^J_{\partial E} (x) = f(x)$, when $f$ is nonconstant. The difficulty of this problem begins with identifying appropriate boundary conditions, expected to be imposed on nonzero measure sets. Regularity results for the graph of $u$ would be dependent on the assumptions imposed on the prescribed nonlocal curvature $f$. To the authors' knowledge this problem is open when the kernel of nonlocal interaction is integrable or not.

\section{Appendix}
\begin{proof}[Proof of Lemma \ref{lem:Figalli 2.1}]
Proof of this Lemma is very similar as in \cite{Figalli}. Here we mention some important steps of the computation using the growth conditions of our kernel and regularity of the boundary. Consider a point $x = (x',x_n) \in \partial \Omega$. As the boundary is $C^{1, \beta}$, there exists a small neighborhood of $x$ where the boundary can be written as the graph of some function $f : \re^{n-1} \to \re$. Without loss of generality, we can assume that $x=0$, $\nabla f(0) = 0 $. Then we define
$$B'_\rho = \{ x' : |x'|< \rho \}, L_\rho = (-\rho, \rho) \times B'_\rho, $$
and
$$\partial \Omega \cap L_\rho = \{(x',f(x'): x' \in B'_\rho \} = (Id \times f)(B'_\rho). $$
 For the $C^1$ convergence, first we define
$$\psi_\varepsilon(t) = - \frac{1}{t^n} \int^\infty_t \phi_\varepsilon(\tau) \tau^{n-1} d \tau $$
Then $$ \dvie (x \psi_\varepsilon) = \phi_\varepsilon(|x|). $$
Using divergence theorem,
$$H^\varepsilon_\Omega (x) = -2\int_{\partial \Omega} \psi_\varepsilon(|x-y|)(x-y) \cdot \nu_y d \sigma $$
Taking derivative in $x$-variable and for some $e \in T_x(\partial \Omega) \cap \mathbb{S}^{n-1}$, we have
$$\nabla H^\varepsilon_\Omega(x) \cdot e = 2 \int_{\partial \Omega}\bigg(\psi_\varepsilon(|x-y|)\nu_y \cdot e + \frac{\psi'_\varepsilon(|x-y|)}{|x-y|}((x-y)\cdot \nu_y)((x-y) \cdot e) \bigg) d \sigma $$

Next we want to evaluate the following quantity, $\norm{\nabla H^\varepsilon_\Omega(x) \cdot e - \nabla H^\eta_\Omega(x) \cdot e}_{C^0}$ when $\varepsilon, \eta \to 0$. As mentioned before, we can assume that $x=0, \nabla f(0) = 0$. Then we have,
$$\nabla H^\varepsilon_\Omega(0) \cdot e = 2 \int_{\partial \Omega}\bigg(\psi_\varepsilon(|y|)\nu_y \cdot e + \frac{\psi'_\varepsilon(|y|)}{|y|}(y\cdot \nu_y)(y \cdot e) \bigg) d \sigma $$
Next we choose a $0<\gamma <\rho$ and define the set $G_\gamma = (Id \times f)(B'_\gamma)$. By the construction of $\psi_\varepsilon, \psi'_\varepsilon$, they converge uniformly in $[t, \infty)$ for any $t>0$. Then we have
\begin{align*}
   & \bigg| \int_{\partial \Omega \setminus G_\gamma }\bigg(\psi_\varepsilon(|y|)\nu_y \cdot e + \frac{\psi'_\varepsilon(|y|)}{|y|}(y\cdot \nu_y)(y \cdot e) \bigg) d \sigma \\
   & \quad - \int_{\partial \Omega \setminus G_\gamma }\bigg(\psi_\eta(|y|)\nu_y \cdot e + \frac{\psi'_\eta(|y|)}{|y|}(y\cdot \nu_y)(y \cdot e) \bigg) d \sigma \bigg| \to 0
\end{align*}
as $\varepsilon, \eta \to 0$. Next inside the set $G_\gamma$, we can write $\nu_y = \frac{(-\nabla f(y'), 1)}{\sqrt{1 + |\nabla f(y')|^2}}$ where $y =(y', f(y'))$ in $G_\gamma$. One can also notice that, as $e \in T_0(\partial \Omega)$ then $e$ does not have any component in $x_n$ direction. Then we write,
\begin{align}\label{eq:cauchy}
  & 2 \int_{G_\gamma}\bigg(\psi_\varepsilon(|y|)\nu_y \cdot e + \frac{\psi'_\varepsilon(|y|)}{|y|}(y\cdot \nu_y)(y \cdot e) \bigg) d \sigma \\ \nonumber
   &= 2 \int_{B'_\gamma} \bigg(-\psi_\varepsilon(\sqrt{|y'|^2+ f^2}) \frac{\nabla f \cdot e}{\sqrt{1 + |\nabla f(y')|^2}} \\ \nonumber
   &\qquad + \frac{\psi'_\varepsilon(\sqrt{|y'|^2+ f^2})}{\sqrt{|y'|^2+ f^2}} \frac{f - \nabla f \cdot y'}{\sqrt{1 + |\nabla f(y')|^2}} (y' \cdot e)\bigg) \sqrt{1 + |\nabla f(y')|^2} \, dy' \\ \nonumber
   &= 2 \int_{B'_\gamma} \bigg(-\psi_\varepsilon(\sqrt{|y'|^2+ f^2}) (\nabla f \cdot e)  + \frac{\psi'_\varepsilon(\sqrt{|y'|^2+ f^2})}{\sqrt{|y'|^2+ f^2}} (f - \nabla f \cdot y') (y' \cdot e)\bigg) \, dy'
\end{align}
By changing $y' \mapsto -y'$ and adding the two expressions, we get the following, same as in \cite{Figalli},
\begin{align}\label{eq:cauchy2}
    & -\int_{B'_\gamma} \psi_\varepsilon (\sqrt{|y'|^2+ f^2}) (\nabla f(y') \cdot e + \nabla f(-y') \cdot e) \, dy' \\ \nonumber
    & + \int_{B'_\gamma} \big( \psi_\varepsilon (\sqrt{|y'|^2+ f(y')^2})- \psi_\varepsilon (\sqrt{|y'|^2+ f(-y')^2}) \big) \nabla f(y') \cdot e  \, dy' \\ \nonumber
    &+ \int_{B'_\gamma} \frac{\psi'_\varepsilon(\sqrt{|y'|^2+ f(y')^2})}{\sqrt{|y'|^2+ f(y')^2}} \big([f(y')-f(-y')](y' \cdot e)- [\nabla f(y') \cdot y' + \nabla f(-y') \cdot y'](y'\cdot e) \big) \, dy' \\ \nonumber
    &+ \int_{B'_\gamma} \bigg( \frac{\psi'_\varepsilon(\sqrt{|y'|^2+ f(y')^2})}{\sqrt{|y'|^2+ f(y')^2}} - \frac{\psi'_\varepsilon(\sqrt{|y'|^2+ f(-y')^2})}{\sqrt{|y'|^2+ f(-y')^2}} \bigg)[f(-y') + \nabla f(-y') \cdot y'](y' \cdot e) \, dy'
\end{align}
Using the regularity of $f$, we have the following estimates,
\begin{equation*}
    |\nabla f(y') \cdot e + \nabla f(-y') \cdot e| \leq C|y'|^\beta, \quad  |\nabla f(-y')| \leq C|y'|^\beta, \quad |f(y')|\leq C|y'|^{1+\beta},
    \end{equation*}
which yield
\begin{equation*}
    |f(y') - f(-y')| \leq C|y'|^{1+\beta}, \quad |\nabla f(y') \cdot y' + \nabla f(-y') \cdot y'|\leq C|y'|^{1+\beta}.
    \end{equation*}
    Moreover,
   \begin{align*} 
    \bigg|\psi_\varepsilon (\sqrt{|y'|^2+ f(y')^2})- \psi_\varepsilon (\sqrt{|y'|^2+ f(-y')^2})&\bigg| \leq C \frac{|f(y')^2 - f(-y')^2|}{\max \{|y'|^{n-\alpha +2}, |y'|^{n+\alpha_1 +2}\}} \\
      \qquad &\leq \frac{C}{\max \{|y'|^{n-\alpha -\beta}, |y'|^{n+\alpha_1 -\beta}\}} 
      \end{align*}
      and also
   \begin{align*} 
    \bigg| \frac{\psi'_\varepsilon(\sqrt{|y'|^2+ f(y')^2})}{\sqrt{|y'|^2+ f(y')^2}} - \frac{\psi'_\varepsilon(\sqrt{|y'|^2+ f(-y')^2})}{\sqrt{|y'|^2+ f(-y')^2}}  \bigg| 
    &\leq C \frac{|f(y')^2 - f(-y')^2|}{\max \{|y'|^{n-\alpha +4}, |y'|^{n+\alpha_1 +4}\}} \\
    \qquad &\leq \frac{C}{\max \{|y'|^{n-\alpha -\beta +2}, |y'|^{n+\alpha_1 -\beta + 2}\}}.
\end{align*}
As $\gamma$ is very small, we can assume that 
\begin{align*}
    \bigg|\psi_\varepsilon (\sqrt{|y'|^2+ f(y')^2})- \psi_\varepsilon (\sqrt{|y'|^2+ f(-y')^2})\bigg|& \leq \frac{C}{ |y'|^{n-\alpha -\beta}} \\
    \bigg| \frac{\psi'_\varepsilon(\sqrt{|y'|^2+ f(y')^2})}{\sqrt{|y'|^2+ f(y')^2}} - \frac{\psi'_\varepsilon(\sqrt{|y'|^2+ f(-y')^2})}{\sqrt{|y'|^2+ f(-y')^2}}  \bigg| &\leq \frac{C}{ |y'|^{n + 2-\alpha -\beta}}.
\end{align*}
Thus if we estimate each term in \eqref{eq:cauchy2}, we get 
$$\bigg|\int_{B'_\gamma} \psi_\varepsilon (\sqrt{|y'|^2+ f^2}) (\nabla f(y') \cdot e + \nabla f(-y') \cdot e) \, dy' \bigg| \leq C\int^\gamma_0 \frac{1}{\kappa^{n-\alpha}} \kappa^\beta \kappa^{n-2} \, d \kappa = C\gamma^{\alpha + \beta -1}$$
and
\begin{align*}
    \bigg|\int_{B'_\gamma} \big( \psi_\varepsilon (\sqrt{|y'|^2+ f(y')^2})&- \psi_\varepsilon (\sqrt{|y'|^2+ f(-y')^2}) \big) \nabla f(y') \cdot e  \, dy'  \bigg|\\
    &\leq C \int^\gamma_0 \frac{1}{\kappa^{n-\alpha -\beta }}\kappa^{\beta}\kappa^{n-2} \, d \kappa \\
    &\leq C \gamma^{\alpha+2\beta -1}.
\end{align*}
Similarly the third and fourth term will also be bounded by $C\gamma^{\alpha + \beta -1}$. Hence 
\begin{align*}
   & \limsup_{\varepsilon, \eta \to 0} \bigg| \int_{ G_\gamma }\bigg(\psi_\varepsilon(|y|)\nu_y \cdot e + \frac{\psi'_\varepsilon(|y|)}{|y|}(y\cdot \nu_y)(y \cdot e) \bigg) d \sigma \\
   & \quad - \int_{ G_\gamma }\bigg(\psi_\eta(|y|)\nu_y \cdot e + \frac{\psi'_\eta(|y|)}{|y|}(y\cdot \nu_y)(y \cdot e) \bigg) d \sigma \bigg| \leq C \gamma^{\alpha + \beta -1}.
\end{align*}
By taking the limit $\gamma \to 0$ we obtain the desired convergence. 
\end{proof}

\end{document}